\documentclass[11pt]{article}

\usepackage{amsmath,amsthm,amssymb, amsfonts,amscd}
\usepackage{mathtools}
\usepackage{bbm}
\usepackage{mathrsfs}
\usepackage{enumitem} 
\usepackage[utf8]{inputenc}
\usepackage{graphicx}

\newtheorem{theorem}{Theorem}[section]
\newtheorem{lemma}[theorem]{Lemma}

\newtheorem{proposition}[theorem]{Proposition}

\theoremstyle{definition}
\newtheorem{definition}[theorem]{Definition}
\newtheorem{remark}[theorem]{Remark}

\newcommand{\Z}{\mathbb{Z}}
\newcommand{\R}{\mathbb{R}}
\newcommand{\N}{\mathbb{N}}

\newcommand{\E}{\mathbb{E}}
\newcommand{\prob}{\mathbb{P}}
\newcommand{\1}{\mathbbm{1}}

\newcommand{\norm}[1]{\left\lVert#1\right\rVert}

\allowdisplaybreaks[1]

\numberwithin{equation}{section}

\usepackage{titlesec}
\titleformat*{\section}{\large\bfseries}
\titlelabel{\thetitle.\quad}
\titleformat{\subsection}[runin]{\normalfont\bfseries}{\thesubsection.}{.5em}{}[.]\titlespacing{\subsection}{0pt}{2ex plus .1ex minus .2ex}{.8em}
\titleformat{\subsubsection}[runin]{\normalfont\itshape}{\thesubsubsection.}{.3em}{}[.]\titlespacing{\subsubsection}{0pt}{1ex plus .1ex minus .2ex}{.5em}

\begin{document}

\title{{Environment viewed from the particle and slowdown for ballistic RWRE in low dimensions}}

\author{
\normalsize{\textsc{Tal Peretz}\footnote{Technion - Israel Institute of Technology. 
E-mail: tal.peretz@campus.technion.ac.il}\ \ \ }}

\maketitle

\begin{abstract}
We consider a random walk in a random environment on $\Z^d$ under ballisticity condition $(T)$. We show the existence of the invariant measure $Q$ with respect to the environment viewed from the particle for $d=2$ and $d=3$, which disproves a conjecture made in \cite{berger2016local} regarding the two-dimensional case. We also prove tail estimates for the Radon-Nikodym derivative $dQ/dP$, where $P$ is the original distribution on the environment. Lastly, we provide nearly sharp tail bounds for regeneration times for $d=3$. 
\newline
MSC 2020 \emph{subject classifications.}  60K37, 82D30.
\end{abstract}

\section{Introduction}
\subsection{Background}
Random Walk in a Random Environment (RWRE) is one the central models for random motion in non-homogeneous media. The model is defined by creating a random environment, and then defining a Markov chain that depends on the environment. More formally, denote by $M_1$ the space of all probability measures on $\mathcal E_d = \{\pm e_j\}_{j=1}^{d}$, where $e_j$ are the canonical coordinate vectors for $\Z^d$, and define $\Omega = M_1^{\Z^d}$. An environment is an element $\omega \in \Omega$, and for $z \in \Z^d$ we denote $\omega(z,e)$ the probability assigned to a coordinate vector $e$ by $\omega(z)$. Given $\omega \in \Omega$ and $z \in \Z^d$, we can define a time-homogeneous Markov chain $\{X_n \}_{n =0}^\infty$ on $\Z^d$ starting at $z$ with transition probabilities
\begin{align*}
P_\omega^z(X_{n+1} = y+ e \vert X_n = y) = \omega(y,e) \: \: \forall y \in \Z^d, e \in \mathcal E_d.
\end{align*}
We refer to $P_\omega^z(\cdot)$ as the quenched measure of the random walk. Let $P$ be a probability measure defined on $\Omega$, and let $E$ be its expectation. Averaging the quenched measure over all of the environments defines another distribution for the random walk
\begin{align*}
\prob^z(\cdot) = E\left[P_\omega^z(\cdot)\right],
\end{align*}
which we refer to as the annealed measure. We write $E_\omega^z[\cdot]$ and $\E^z[\cdot]$ for the expectation with respect to the quenched and annealed measure, respectively. 
\subsubsection{Ballisticity}
In this paper, we will consider RWRE satisfying
\begin{itemize}
\item $P$ is an i.i.d. measure: $P = \nu^{\Z^d}$ for some measure $\nu$ on $M_1$.
\item $P$ is uniformly elliptic: there exists $\kappa > 0$ such that for all $x\in \Z^d$, 
\begin{align*}
P\left(\min_{e \in \mathcal E_d} \omega(x,e)>\kappa \right) = 1.
\end{align*}
\end{itemize}
 Under these conditions, in \cite{sznitman1999law, zeitouni2004, zerner2002non} the authors proved a law of large numbers:
\begin{align} \label{eq: LLN}
\lim_{n \to \infty}\frac{X_n}{n}= \mathbbm v, \quad \prob^0-\text{a.s.}
\end{align}
However, it still remains open whether $\mathbbm v$ is deterministic and non-vanishing. When $\mathbbm v$ is a $\prob$-a.s. non-zero constant, we say the random walk is ballistic. There has been a research focus in studying ballistic RWRE and answering questions surrounding \eqref{eq: LLN}. In \cite{sznitman2001class, sznitman2002effective}, Sznitman introduced a criterion for ballisticity. For $A \subset \Z^d$, define the hitting time
\begin{align*}
 T_A =\inf \{n \geq 0:  X_n \in A \}.
\end{align*}
For $L>0$ and $\ell \in S^{d-1}= \{z \in \R^d: \vert z \vert =1\}$, by abuse of notation we write
\begin{align*}
T_L = T^{(\ell)}_L = \inf \{n \geq 0:  \langle X_n,   \ell \rangle  \geq L \}.
\end{align*}
\begin{definition} \label{def: condt}(Sznitman \cite{sznitman2002effective})
We say $P$ satisfies condition $(T)$ in direction $\ell_0 \in S^{d-1}$ if for every $\ell \in S^{d-1}$ in some open neighborhood of $\ell_0$ there exist positive constants $C$ and $c$ such that
\begin{align*}
\prob^0( T^{(-\ell)}_L <  T^{(\ell)}_L) \leq Ce^{-cL},
\end{align*}
for every $L>0$.
\end{definition} 
Condition $(T)$ has two other equivalent formulations, one being an effective criterion, see \cite{guerra2020proof,sznitman2002effective}, and the other a moment assumption on regeneration times, see \cite{sznitman2001class}. This condition was introduced in order to guarantee ballisticity, which the following theorem justifies. 
\begin{theorem} [Sznitman \cite{sznitman2002effective}] \label{thm: condt-lln}
If condition $(T)$ holds for some direction $\ell_0$, then
\begin{align*}
\lim_{n \to \infty} \frac{X_n}{n} = \mathbbm v
\end{align*}
for some deterministic constant $\mathbbm v \in \R^d$ that satisfies $\langle \mathbbm v , \ell_0 \rangle >0$.
\end{theorem}
At the moment, condition $(T)$ is the most general assumption that implies ballisticity, and it is even conjectured to be equivalent to ballisticity, see \cite{sznitman2002effective}. There has been a research focus in generalizing this condition and relaxing the exponential decay assumption in Definition \ref{def: condt}, see \cite{berger2014effective, drewitz2011conditions, drewitz2012quenched, guerra2020proof}.

\subsubsection{Environment Viewed from the Particle }
One of the major tools in studying RWRE is an auxiliary Markov chain on the space of environments $\Omega$ called the environment viewed from the particle
\begin{align*}
\overline \omega_n = \sigma_{X_n}\omega = \omega(X_n + \cdot) \: \: \text{for}\: \:  n \in \N_0,
\end{align*} 
where $\sigma_x$ for $x \in \Z^d$ denotes the canonical shift on $\Omega$. One of the main technical difficulties with RWRE is that the random walk is not a Markov chain under the annealed distribution. An advantage of $\overline \omega_n$ is that under $\prob$ it is a Markov chain with compact state space $\Omega$, initial distribution $P$, and transition kernel
\begin{align*}
\mathcal R f(\omega) = \sum_{e \in \mathcal E_d} \omega(x,e) f (\sigma_e\omega)
\end{align*}
for $f: \Omega \to \R$ bounded and measurable. We say that a probability measure $Q$ on $\Omega$ is invariant with respect to the point of view of the particle if
\begin{align*}
\int_\Omega \mathcal R g(\omega) dQ(\omega) = \int_\Omega g(\omega) dQ(\omega),
\end{align*}
for every bounded continuous $g: \Omega \to \R$. Invariant measure are not difficult to find, but are only useful if they are equivalent to $P$. When such a measure exists, we say the static and dynamic points of view of the environment are equivalent. Such invariant measures have been used to prove central limit theorems and law of large numbers, see \cite{alili1999asymptotic, bolthausen2002static, bolthausen2012ten,drewitz2014selected,kozlov1985method, zeitouni2004}. Consequently, there is major interest in finding invariant measures which are equivalent to $P$. 

In \cite{bolthausen2002static}, Bolthausen and Sznitman showed that $Q$ exists for certain high-dimensional ballistic RWRE. In \cite{berger2016local}, Berger, Cohen and Rosenthal extended their result by proving that for $d \geq 4$ and under condition $(T)$, the static and dynamic points of view of the environment are equivalent. For both papers, the methods relied on bounding the number of intersections of independent random walks in the same environment. Intersection estimates often appear in RWRE, since small number of intersections is a proxy for mixing of the random walk. Due to the small number of intersections in high dimensions, \cite{berger2016local} and  \cite{bolthausen2002static} were able to prove the existence of $Q$ for $d \geq 4$. In  \cite{berger2016local}, the authors also conjectured that for $d=2$, where there are many intersections, there is no invariant measure equivalent to $P$. 
\subsubsection{Slowdown Estimates}

Another tool in studying ballistic RWRE is regeneration times, which where introduced in \cite{sznitman1999law} for the purpose of proving \eqref{eq: LLN}.

\begin{definition}
Fix $\ell \in S^{d-1}$. We define regeneration times for $\{ X_n \}$ in direction $\ell$ as $\tau_0= 0$, and for $j \in \N$
\begin{align*}
\tau_{j} = \inf \{n > \tau_{j-1}: \langle X_k, \ell \rangle < \langle X_n, \ell \rangle \: \text{for all} \: k<n \: \text{and}\: \langle X_k, \ell \rangle \geq \langle X_n, \ell \rangle \: \text{for all} \: k >n \},
\end{align*}
where we use the convention $\inf \varnothing = \infty$.
\end{definition}
The following theorem summarizes the main properties of regeneration times we will use in this paper.

\begin{theorem}[Sznitman, Zerner \cite{sznitman2002effective},\cite{sznitman1999law}] \label{thm: regen}
Assume $P$ is uniformly elliptic, i.i.d. and satisfies condition $(T)$ in direction $\ell$. Then 
\begin{enumerate}
\item $\prob$-a.s. there exist infinitely many regeneration times $\tau_1 < \tau_2 <\cdots$.
\item Under $\mathbb{P}$, the sequence$\{(\tau_{k+1}-\tau_k, X_{\tau_{k+1}} - X_{\tau_k}) \}_{k \in \N}$ is i.i.d., and independent of $(\tau_1, X_{\tau_1})$.
\item There exists a constant $C>0$ such that for all $u>0$, $\prob(\tau_2 - \tau_1 > u) \leq C \prob(\tau_1 > u )$.
\end{enumerate}  
\end{theorem}
The theorem provides an i.i.d. structure to the random walk, and so questions regarding limit theorems reduces to understanding the moments of the regeneration times. For example, Sznitman and Zerner in \cite{sznitman1999law} showed the random walk is ballistic when $\E[\tau_2 - \tau_1] < \infty$, and Sznitman proved in \cite{sznitman1999slowdown} an annealed invariance principle when $\E[(\tau_2 - \tau_1)^2] <\infty$. Berger and Zeitouni \cite{berger2008quenched}, and independently Rassoul-Agha and Sepp\"{a}l\"{a}inen \cite{Rassoul2009functional}, proved a quenched invariance principle  for ballistic RWRE assuming high enough moments of the regeneration times. For this reason, there has been much effort in understanding the tail behavior of regeneration times. In \cite{sznitman1999slowdown, sznitman2000slowdown}, Sznitman showed that they are dominated by traps, atypical pockets in the environment where the random walk is likely to spend a long time. We define the local drift at $x$ for the environment $\omega$ by
\begin{align*}
d(x,\omega) = E_\omega^x[X_1 - X_0] = \sum_{e \in \mathcal E_d}\omega(x,e)e.
\end{align*}
Following \cite{sznitman2000slowdown}, we say $P$ is nestling if $0$ is in the interior of $K_0$, where $K_0$ be the convex hull of the support of the law of $d(0,\omega)$. Sznitman showed that when $P$ is nestling,
\begin{align} \label{eq: slowdown-lower}
\prob^0(\tau_1 > u) > Ce^{-c(\log u)^d}
\end{align}
by creating the na\"{i}ve trap. Berger was able to prove a nearly matching upper bound in \cite{berger2012slowdown}: for $d\geq 4$ and for $P$ satisfying condition $(T)$, we have
\begin{align}\label{eq: berger-slowdown}
\prob^0(\tau_1 > u) < Ce^{-(\log u)^{\alpha}}
\end{align}
for every $\alpha <d$.
\subsection{Main results}
In our first result we prove the existence of a unique invariant measure $Q$ for $\Z^2$ and $\Z^3$, which disproves the conjecture made in \cite{berger2016local}.
\begin{theorem}\label{thm: invariant}
Let $d  \geq 2$, and assume $P$ is uniformly elliptic, i.i.d. and satisfies condition $(T)$. Then there exists a unique invariant probability measure $Q$ on $\Omega$ which is invariant with respect to $\mathcal R$ and is equivalent to $P$. Furthermore, we have for $P$-almost every $\omega \in \Omega$
\begin{align*}
\lim_{n \to \infty} \sum_{x \in \Z^d} \left\vert P_\omega^0(X_n=x) - \prob^0(X_n = x) \frac{dQ}{dP}(\sigma_x \omega) \right\vert =0.
\end{align*}
\end{theorem}
\begin{remark}
This theorem was proved in \cite{berger2016local} for $d \geq 4$. For the one-dimensional case, Alili proved in \cite{alili1999asymptotic} the existence of $Q$ whenever the random walk is ballistic.
\end{remark}
Our next result draws a connection between the slowdown effect and the environment viewed from the particle. Heuristically, for an invariant measure $Q$ which is equivalent to $P$ to exist, the random walk should be well mixing in the environment. However, this will not be the case if the random walk spends a large amount of time in a small region. The following theorem validates this intuition by showing that the tail behavior of $dQ/dP$ is dominated by traps in the environment.
\begin{theorem}\label{thm: radon-quantitative}
Assume $P$ is uniformly elliptic, i.i.d. and satisfies condition $(T)$, and let $Q$ be the invariant measure from Theorem \ref{thm: invariant}. For $d \geq 3$, for every $\alpha < d$ there exist constants $C,c>0$ such that
\begin{align*}
\forall u >  0, \: P\left(dQ/dP >  u \right)\leq Ce^{-c(\log u)^\alpha}.
\end{align*}
For $d \geq 2$ and assuming $P$ is nestling, there exist constants $C,c>0$ such that
\begin{align*}
\forall u > 0, \:  P\left(dQ/dP >  u \right)\geq  Ce^{-c (\log u)^d}.
\end{align*}
\end{theorem}
Previously, $\cite{berger2016local}$ showed that $dQ/dP$ has stretched exponential tails for $d \geq 4$. Our bounds improve on this and are new in all dimensions. A corollary of our intersection estimates, Proposition \ref{prop: intersection}, is a new bound on the tail of regeneration times, which settles a conjecture made in \cite{berger2012slowdown}.
\begin{theorem} \label{thm: slowdown}
Let $d = 3$ and assume $P$ is uniformly elliptic, i.i.d. and satisfies condition $(T)$. Then for every $\alpha < 3$ there exists $C>0$ such that
\begin{align*}
\forall u > 0, \: \prob^0(\tau_1 > u) \leq Ce^{-(\log u)^\alpha}.
\end{align*}
\end{theorem}
\begin{remark}
An upper bound for $\prob^0(\tau_1>u)$ was proven for $d=1$ in \cite{dembo1996tail} and for $d \geq 4$ in \cite{berger2012slowdown}, leaving $d=2$ as the only open case.
\end{remark}
\subsection{Remarks on the proofs and structure of the paper}
This article provides two types of results. The first kind, Theorems \ref{thm: invariant} and \ref{thm: slowdown}, extends previous findings to dimensions $d=2$ and $d=3$. The main ingredient for their proofs is Proposition \ref{prop: heatkernel}, a quenched local limit theorem for large boxes. The proof of this proposition is based on a martingale difference argument and relies on bounding the expected number of intersections of two independent random walks. The key step to proving Theorems \ref{thm: invariant} and \ref{thm: slowdown} is Proposition \ref{prop: intersection}, which provides bounds on the number of intersections in lower dimensions. Given this estimate, the proof of Proposition \ref{prop: heatkernel} is similar to \cite{berger2016local}, with two exceptions. First, since we are in lower dimensions the intersection bounds we use are larger than those in \cite{berger2016local}. Second, we provide nearly sharp estimates on the probability the environment has atypically large heat kernels, which we will need for the proof of Theorem \ref{thm: radon-quantitative}.

The second type of results, Theorem \ref{thm: radon-quantitative}, are sharp bounds for the fluctuations of $dQ/dP$ for $d \geq 3$. The proof entails approximating $dQ/dP$ as a sum of heat kernels and then studying their fluctuations. For the upper bound, we estimate the heat kernel for finite boxes, see Lemma \ref{lem: smallbox-upper}. For the lower bound, we place a trap at the origin and show that the probability of hitting the origin is large in this environment, see Lemma \ref{lem: smallbox-lower}.

The paper is structured in the following way. In Section \ref{sec: preliminaries} we define notation and recall previous results which will be used throughout the paper. In Section \ref{sec: intersections} we bound the number of intersection of two random walks. In Section \ref{sec: heat-kernel} we state the semi-local limit theorem and prove Theorem \ref{thm: invariant}. The rest of the section we use the new heat kernel bounds to improve the previous intersection estimates, see Lemma \ref{lem: intersections-upgrade}, which will be used to prove Theorem \ref{thm: slowdown}. In Section \ref{sec: radon-quantitative} we prove Theorem \ref{thm: radon-quantitative}. Since the proofs of Proposition \ref{prop: heatkernel} and Theorem \ref{thm: slowdown} are similar to those in \cite{berger2012slowdown, berger2016local}, we leave it to the appendix.  

A few remarks about the conventions in this paper:
\begin{itemize}
\item Throughout the rest of this text, we denote $c,c',C,C$ to be generic numbers which change from line to line, while numbered constants will be fixed throughout the text.
\item When assuming $P$ satisfies condition $(T)$ in direction $\ell$, without any loss we will assume $\ell= e_1$. In this case, by Theorem \ref{thm: condt-lln} the limiting velocity $\mathbbm v$ satisfies $\langle \mathbbm v, e_1 \rangle >0$.
\end{itemize}

\section{Preliminaries} \label{sec: preliminaries}
In this section we introduce notation and previous results which will be used throughout the paper.\\

For positive sequences $a_n$ and $b_n$, we write  $a_n \ll b_n$ if $\lim_{n \to \infty} a_n/b_n = 0$, $a_n = \xi(b_n)$ if $ \lim_{n \to \infty }b_n/a_n = 0$, $a_n = O(b_n)$ if $\limsup_{n \to \infty} a_n /b_n < \infty$ and $a_n \asymp b_n $ if $0 < \liminf_{n \to \infty} a_n/b_n \leq \limsup_{n \to \infty} a_n/b_n <\infty.$  Note that in our notation, $n^{-\xi(1)}$ is a sequence going to zero faster than any polynomial of $n$. For $j \in \N$, define the sequence $R_j(N) = \exp((\log N)^{\frac{j+2}{j+3}})$. We will often use the fact that for any fixed $\epsilon, \kappa>0$ and $j \in \N$, we have
\begin{align} \label{eq: R-slow}
\log N \ll R^{\kappa}_j(N) \ll R_{j+1}(N) \ll N^{\epsilon}
\end{align}
as $N \to \infty$.\\

Denote $\norm{\cdot}_p$ for $p \in \{1,\infty \}$ to be the usual $\ell^p$ norms, while $\vert \cdot \vert$ will denote the Euclidean norm. For $z \in \Z^d$ and $r > 0$, define
\begin{align*}
B(z,r) = \{x \in \Z^d: \norm{ z-x}_\infty  \leq r \}, \qquad B_2(z,r) = \{x \in \Z^d: \vert x -z \vert \leq r \}.
\end{align*} 
For $L \in \Z$, define the hyperplane 
\begin{align*}
H_L = \{z \in \Z^d: \langle z, e_1 \rangle = L \}.
\end{align*} 
For sets $A,B \subset \Z^d$, define $\text{dist}(A,B) = \min\{ \norm{x-y}_1: x \in A, y \in B \}$ and $z + A = \{x+z : x\in A \}$. \\

Define the asymptotic direction of the random walk
\begin{align*}
\vartheta= \lim_{n \to \infty}\frac{X_n}{ \vert X_n  \vert } \in S^{d-1}.
\end{align*}
Note that by our earlier assumption $\langle \vartheta, e_1 \rangle > 0$. For $z \in \Z^d$ and $N \in \N$, define the parallelogram of side-length $N$ and center $z$ to be
\begin{align*}
\mathcal P(z,N) = \left \{ x\in \Z^d: \vert \langle x-z , e_1 \rangle \vert < N^2, \norm{x - z -\vartheta \cdot \frac{\langle x-z,e_1 \rangle}{\langle \vartheta, e_1 \rangle} }_\infty<N R_5(N) \right\}.
\end{align*}
We also define the middle third of $ \mathcal P(z,N)$
\begin{align*}
\tilde{ \mathcal {P}}(z,N) = \left \{ x\in \Z^d: \vert \langle x-z , e_1 \rangle \vert < N^2/3, \norm{x - z -\vartheta \cdot \frac{\langle x-z,e_1 \rangle}{\langle \vartheta, e_1 \rangle} }_\infty<N R_5(N) /3\right\},
\end{align*}
the boundary of $\mathcal P(z,n)$
\begin{align*}
\partial \mathcal P(z,N) = \{x \in \Z^d \setminus \mathcal P(z,n): \exists y \in \mathcal P(z,N) \: \text{s.t.} \: \norm{x-y}_1=1 \},
\end{align*}
and the right boundary of $\mathcal P(z,n)$
\begin{align*}
\partial^+ \mathcal P(z,N) = \{x \in \partial \mathcal P(z,n): \langle x-z,e_1 \rangle = N^2 \},
\end{align*}
see Figure \ref{fig: parallelogram}.
\begin{figure}[t!]  \centering
    \includegraphics[scale=0.50]{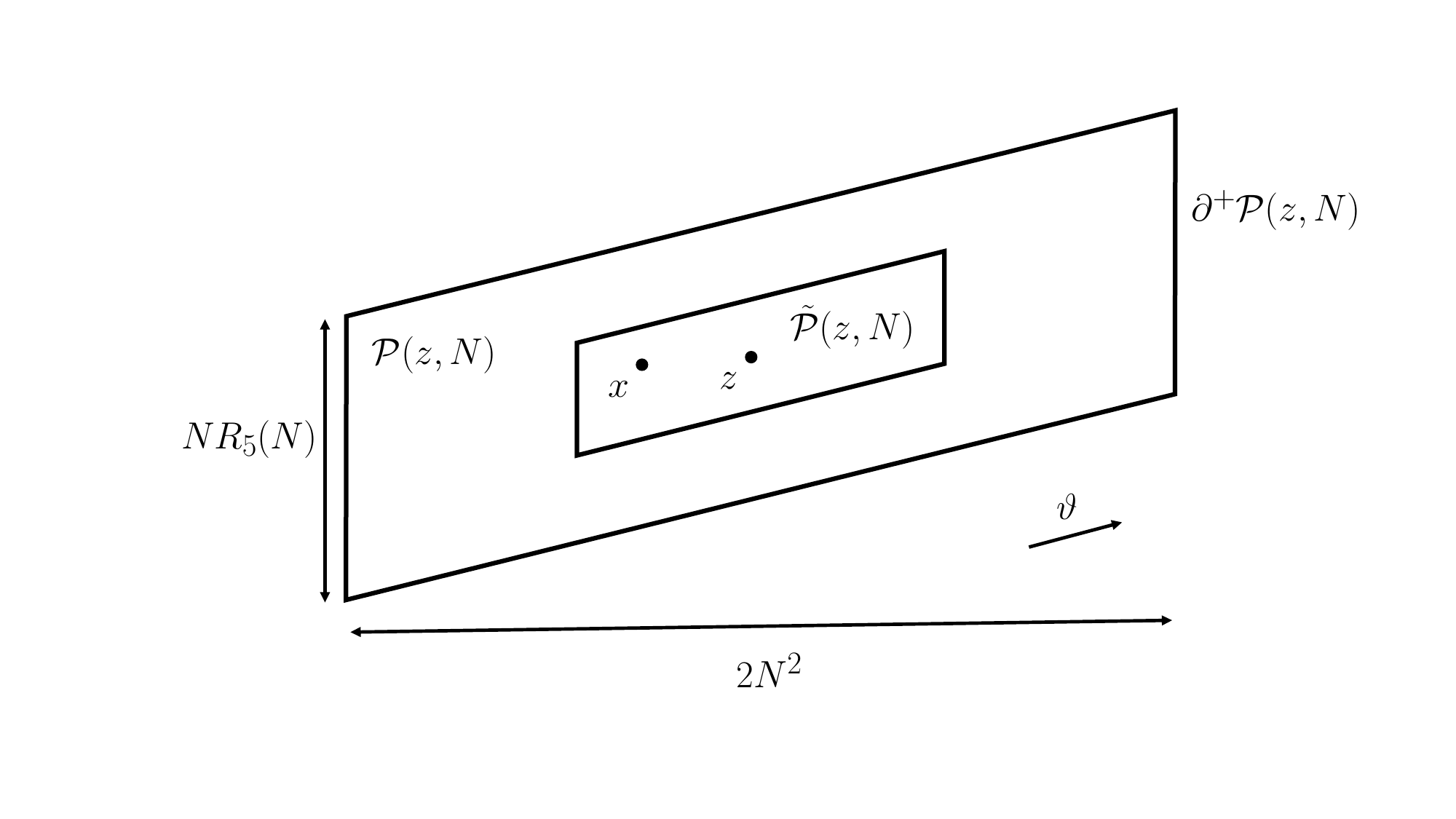}
    \caption{\textit{Assume the random walk satisfies condition $(T)$ in direction $e_1$, which implies it is ballistic in direction $\vartheta $ satisfying $\langle \vartheta, e_1 \rangle > 0$. For any $x \in \tilde{\mathcal P}(z,N)$, the random walk starting at $x$ will exit the parallelogram $ \mathcal P(z,N)$ through $\partial^+ \mathcal P(z,N)$ with very high probability.}}
          \label{fig: parallelogram}
\end{figure}
We will need tail bounds for regeneration times that hold for all relevant dimensions.
\begin{lemma}[{\cite[Theorem 3.5]{sznitman2000slowdown}}]\label{lem: sub-slowdown}
Let $d \geq 2$. Assume $P$ is uniformly elliptic, i.i.d. and satisfies condition $(T)$. Then for any $\alpha < 1 + \frac{d-1}{3d}$, there exist constants $C,c>0$ such that
\begin{align*}
\forall u >0, \: \prob(\tau_1 > u) \leq C e^{-c(\log u)^\alpha}.
\end{align*}
\end{lemma}
\begin{remark}
This bound is not sharp, as can be seen in \eqref{eq: slowdown-lower}. However it holds for $d \geq 2$, which we will need for proving Theorems \ref{thm: invariant} and \ref{thm: slowdown}. To prove Theorem \ref{thm: radon-quantitative}, we will use \eqref{eq: berger-slowdown} and Theorem \ref{thm: slowdown} to assume Lemma \ref{lem: sub-slowdown} holds for $\alpha<d$ for $d \geq 3$.
\end{remark}
We will need estimates for the probability that the random walk does not exit $\mathcal P(z,N)$ through the right, or spends a long time in the parallelogram. 
\begin{lemma}[{\cite[Lemma 3.4]{berger2016local}}] \label{lem: exit-estimates}
Let $d \geq 2$, and assume $P$ is uniformly elliptic, i.i.d. and satisfies condition $(T)$. For every $z \in \tilde {\mathcal P}(0,N)$, every $\alpha <1 + \frac{d-1}{3d}$, and every $j \in \N$, there exist constants $C,c>0$ such that
\begin{align*}
\prob^z(T_{\partial \mathcal P(0,N)} \neq T_{\partial^+ \mathcal P(0,N)}) \leq C\exp(-c R_5(N))
\end{align*}
 and
\begin{align*}
\prob^z \left( \vert T_{\partial \mathcal P(0,N)} - \E^zT_{\partial \mathcal P(0,N)} \vert > R_{j+1}(N)N\right) \leq C\exp(-c \log(R_j(N))^\alpha)
\end{align*}
for every $n \in \N$.
\end{lemma}
\begin{remark}
The bounds in \cite[Lemma 3.4]{berger2016local} are stated for $d \geq 4$ and have a different sub-exponential bound. By going over their proof and using Lemma \ref{lem: sub-slowdown}, we obtain Lemma~\ref{lem: exit-estimates}. For example, following their proof, we have
\begin{align*}
\prob^z \left( \vert T_{\partial \mathcal P(0,N)} - \E^zT_{\partial \mathcal P(0,N)} \vert > R_{j+1}(N)N\right) &\leq Ce^{-cR^2_{j+1}(N)/R^2_j(N)} \\
&+ \prob(\exists \: 1 \leq k \leq N^2 : \tau_k - \tau_{k-1} > R_j(N)).
\end{align*}
We can now apply Lemma \ref{lem: sub-slowdown} to bound the second term on the right-hand-side.
\end{remark}
\begin{lemma}[{\cite[Lemma 2.16]{berger2016local}}] \label{lem: exit-ball}
Assume $P$ is uniformly elliptic, i.i.d. and satisfies condition $(T)$. Then there exist constants $C,c>0$ such that
\begin{align*}
\prob^z(\norm{X_n - \E^z[X_n]}_\infty > R_5(n) n^{1/2}) \leq C\exp(-c R_5(n)),
\end{align*}
for every $n \in \N$.
\end{lemma}
The next two lemmas provide us annealed heat kernel bounds.
\begin{lemma} [{\cite[Lemma 2.14]{berger2016local}}] \label{lem: ann-kernel1}
Assume $P$ is uniformly elliptic, i.i.d. and satisfies condition $(T)$. There exists a constant $C>0$ such that for all $n \in \N$, and $x,y,z,w \in \Z^d$ such that $\norm{x-y}_1= 1$ and $\norm{z-w}_1 =1$
\begin{align*}
\prob^z(X_n = x) &\leq C n^{-d/2} \\
\vert \prob^z(X_n = x) - \prob^z(X_{n+1} = y)\vert & \leq C n^{-(d+1)/2}\\
\vert \prob^z(X_n = x) - \prob^w(X_{n+1} = x)\vert & \leq C n^{-(d+1)/2}.
\end{align*}
\end{lemma}
\begin{lemma}[{\cite[Lemma 3.3]{berger2016local}}] \label{lem: ann-kernel2}
Assume $P$ is uniformly elliptic, i.i.d. and satisfies condition $(T)$. There exists a constant $C>0$, such that for all $z_1 \in \Z^d$, $N \in \N$ and $z \in \tilde{ \mathcal P}(z_1,N )$, we have
\begin{enumerate}
\item For every $m \in \N$ and $x \in \partial^+ \mathcal P(z_1, N)$
\begin{align*}
\prob^z(T_{\partial \mathcal P(z_1,N)} = m, X_{T_{\partial \mathcal P(z_1,N)}} = x) < C N^{-d}.
\end{align*}
\item For every $m \in \N$ and $x,y \in \partial^+ \mathcal P(z_1, N)$ such that $\norm{x-y}_1 = 1$
\begin{align*}
\vert \prob^z(T_{\partial \mathcal P(z_1,N)} = m, X_{T_{\partial \mathcal P(z_1,N)}} = x) -\prob^z(T_{\partial \mathcal P(z_1,N)} = m, X_{T_{\partial \mathcal P(z_1,N)}} = y)  \vert < C N^{-d-1}.
\end{align*}
\item For every $m \in \N$, every $x \in \partial^+ \mathcal P(z_1, N)$ and every $e \in \mathcal E_d$ 
\begin{align*}
\vert \prob^z(T_{\partial \mathcal P(z_1,N)} = m, X_{T_{\partial \mathcal P(z_1,N)}} = x) -\prob^{z+ e}(T_{\partial \mathcal P(z_1,N)} = m+1, X_{T_{\partial \mathcal P(z_1,N)}} = x)  \vert < C N^{-d-1}.
\end{align*}
\end{enumerate}

\end{lemma}
\section{Intersections of Random Walks} \label{sec: intersections}
All of our results will rely on bounds for the expected number of intersections of two independent random walks in the same environment. Let $X^{(1)}$ and $X^{(2)}$ be two independent random walks in the same environment starting at the origin, i.e. their joint law is given by
\begin{align*}
P^0_\omega \otimes P^0_\omega (X^{(1)} \in \cdot, X^{(2)} \in \cdot) = P^0_\omega(X^{(1)} \in \cdot)  P^0_\omega(X^{(2)} \in \cdot).
\end{align*}
For $n \in \N$, define 
\begin{align*}
\mathcal I_n = \sum_{\substack{z \in \Z^d \\ \norm{z}_1 \leq n}  } \1 \{X^{(1)} \: \: \text{visits} \: \: z\}\1 \{X^{(2)} \: \: \text{visits} \: \: z\}
\end{align*}
to be the number of intersection points of $X^{(1)}$ and $X^{(2)}$ in $\{z \in \Z^d: \norm{z}_1 \leq n\}$. The only part of the proof of the existence of an invariant measure in \cite{berger2016local} which assumes $d \geq 4$ is essentially bounding $E^0_\omega\otimes E^0_\omega  [\mathcal I_n] $. In high dimensions, the probability two random walks ever intersect is uniformly bounded away from zero, and hence $\mathcal I_n$ is dominated by a geometric random variable. In particular, the number of intersections is very small. This is not the case for $d=2$, and so we use an argument from the proof of \cite[Lemma 4.2]{berger2008quenched}. There, the authors bounded $E[E^0_\omega\otimes E^0_\omega  [\mathcal I_n] ]$ by showing that while two random walks might intersect often, the time between each intersection is long. To bound the martingale differences that arise in the proof of Proposition \ref{prop: heatkernel}, we will need to bound the tails of the quenched mean of $\mathcal I_n$. We will also need to control these tail bounds for the proof of Theorem \ref{thm: radon-quantitative}. We adjust the the arguments from \cite{berger2008quenched} accordingly, taking advantage of Lemma \ref{lem: sub-slowdown} to control the tail probability.
\begin{proposition}\label{prop: intersection}
Assume $d\geq 2$, and assume $P$ is uniformly elliptic, i.i.d. and satisfies condition $(T)$. Then for any $\alpha < 1+ \frac{d-1}{3d}$ and any $\epsilon>0$, there exist constants $C,c>0$ such that
\begin{align*}
P\left(E^0_\omega\otimes E^0_\omega  [\mathcal I_n] > n^{1/2+\epsilon}  \right) < C\exp(-c (\log n)^\alpha)
\end{align*}
for every $n \in \N$.
\end{proposition}
\begin{proof}
 Recall that $\mathbbm v$ is the limiting velocity of the random walk, and define
\begin{align*}
G(t) = \sum_{i,j=0}^\infty  \1_{\{X^{(1)}_i=X^{(2)}_j\}} \1_{\{ \langle X^{(1)}_i,\mathbbm v \rangle \in [t-0.5,t+0.5) \}}.
\end{align*}
We have
\begin{align*}
E^0_\omega\otimes E^0_\omega  [\mathcal I_n] &\leq \sum_{\substack{z \in \Z^d \\ \norm{z}_1 \leq n }}E^0_\omega\otimes E^0_\omega \left[\sum_{i,j=0}^\infty \1_{ \{ X_i^{(1)} = X_j^{(2)} = z \}} \right]\\
&\leq \sum_{t=-n}^n \sum_{\substack{ z \in \Z^d \\ \langle z,\mathbbm v  \rangle \in [t-0.5,t+0.5)}} E^0_\omega\otimes E^0_\omega  \left[\sum_{i,j=0}^\infty\1_{ \{ X_i^{(1)} = X_j^{(2)} = z \}} \right]\\
&= \sum_{t=-n}^n E^0_\omega\otimes E^0_\omega \left[\sum_{i,j=0}^\infty \1_{ \{ X_i^{(1)} = X_j^{(2)}\} }\1_{ \{ \langle X^{(1)}_i, \mathbbm v \rangle \in [t-0.5,t+0.5) \}} \right]\\
&= \sum_{t=-n}^n E^0_\omega\otimes E^0_\omega [G(t)] \\
&= \sum_{t=-n}^n E^0_\omega\otimes E^0_\omega [G(t) \1_{\{G(t) \neq 0\}}] .
\end{align*}
Let $\{ \tau_k^{(i)} \}_{k \in \N}$ denote the respective regeneration times of $X^{(i)}$, and define 
\begin{align*}
Y^{(i)}_n = \max\{ \tau_{k+1}^{(i)}-\tau_k^{(i)}: 1\leq k \leq n\}, \quad Y_n = \max \{Y^{(1)}_n , Y^{(2)}_n  \}.
\end{align*}
Since $G(t)$ is the number of pairs of intersection times in
\begin{align*}
\{ z \in \Z^d: \langle z ,\mathbbm v \rangle \in [t-0.5,t+0.5)  \},
\end{align*}
 it is bounded by the length of the $X^{(1)}$ regeneration interval containing $t$, multiplied by the length of the $X^{(2)}$ regeneration interval containing $t$. Hence
\begin{align*}
G(t) \leq Y^{(1)}_n \cdot Y^{(2)}_n \leq Y_n^2,
\end{align*}
 and so
\begin{align*}
E^0_\omega\otimes E^0_\omega [\mathcal I_n] \leq   E^0_\omega \otimes E^0_\omega \left[Y_n^2 \sum_{t=-n}^n\1_{\{ G(t) \neq 0 \}} \right].
\end{align*}
To bound the sum inside the expectation, we will need to introduce more notation. Define the variables $\{\psi_n \}_{n \in \N}$ and $\{\theta_n \}_{n \in \N}$ inductively: $\theta_0 = 0$, $\psi_1 = \max \{ \tau_1^{(1)} , \tau_1^{(2)} \}$, and for $n \geq 1$
\begin{align*}
\theta_n = \min \{k > \psi_n: G(k) \neq 0 \}, \quad \psi_{n+1} = \max \{ \tau^{(1)}(\theta_n) , \tau^{(2)}(\theta_n) \},
\end{align*}
where
\begin{align*}
\tau^{(i)}(k) = \min \{ \langle X_{\tau^{(i)}_m}, \mathbbm v \rangle:   \langle X_{\tau^{(i)}_m}, \mathbbm v \rangle > k+1, m \in \N   \}.
\end{align*}
Since we are under condition $(T)$, $\tau^{(i)}(k) $ is well-defined. Define $j_n = \theta_n - \psi_n$, $h_n = \psi_{n}-\theta_{n-1}$ and
\begin{align*}
K_n = \min \left\{m \in \N: \sum_{i=1}^m j_i > n \right\},
\end{align*}
and note that
\begin{align*}
\sum_{t=-n}^n \1_{\{G(t) \neq 0\}} \leq K_n \cdot Y_n.
\end{align*}
It follows that $E^0_\omega\otimes E^0_\omega[\mathcal I_n] < E^0_\omega\otimes E^0_\omega[Y_n^3 \cdot K_n]$, and so we are left to derive upper bounds for $K_n$ and $Y_n$. From \cite[(4.19)]{berger2008quenched}, there exists a constant $c_1>0$ such that for any $\delta_1>0$ we have
\begin{align} \label{eq: int-lower} 
E[P^0_\omega \otimes P^0_\omega(j_n > k\vert j_1,\ldots, j_{n-1}, h_1,\ldots, h_n)] \geq c_1/ k^{1/2+\delta_1}.
\end{align}
Hence for $m \in \N$, we have
\begin{align*}
E[P^0_\omega \otimes P^0_\omega (K_n > m)] &= E \left[P^0_\omega \otimes P^0_\omega\left ( \sum_{i=1}^m j_i < n \right ) \right] \\
&\leq E[P^0_\omega \otimes P^0_\omega (\forall 1 \leq i \leq m, j_i < n  )]  \\
&\leq (1-c_1/n^{1/2+\delta_1})^m  \\
&\leq  \exp\left(-c_1 \frac{m}{n^{1/2+\delta_1}}\right),
\end{align*}
where the second inequality follows from \eqref{eq: int-lower}. For $\delta_2 > \delta_1 > 0$, define the event 
\begin{align*}
B_n = \left\{\omega \in \Omega:  P^0_\omega\otimes P^0_\omega (K_n > n^{1/2 + \delta_2}) <  \exp\left(-\frac{c_1}{2} n^{\delta_2 - \delta_1} \right) \right\}.
\end{align*}
From Markov's inequality we have 
\begin{align} \label{eq: int-bad-1}
P(B^c_n) < \frac{ E[P^0_\omega\otimes P^0_\omega(K_n>n^{1/2 + \delta_2})] } {\exp\left(-\frac{c_1}{2}n^{\delta_2 - \delta_1}  \right)} \leq \exp\left(-\frac{c_1}{2}n^{\delta_2-\delta_1}\right).
\end{align} 
To bound $Y_n$, define the event 
\begin{align*}
A_n= A_n(j) = \left\{\omega \in \Omega:  P^0_\omega \otimes P^0_\omega (Y_n > R_j(n)) < C\exp\left({-c(\log n)^{\alpha (j+2)/(j+3)}} \right) \right\}.
\end{align*}
From Lemma \ref{lem: sub-slowdown} and Theorem \ref{thm: regen}, for every $\alpha < 1+ \frac{d-1}{3d}$ there exist constants $C,c>0$ such that 
\begin{align} \label{eq: int-bad-2}
P(A^c_n) \leq C\exp\left({-c(\log n)^{\alpha (j+2)/(j+3)}} \right) 
\end{align} for all $n \in \N$. We decompose the expectation
\begin{align*}
E^0_\omega \otimes E^0_\omega [\mathcal I_n] &\leq E^0_\omega \otimes E^0_\omega [Y^3_n \cdot K_n]  \\
&\leq E^0_\omega \otimes E^0_\omega [Y^3_n \cdot K_n \1_{K_n \leq  n^{1/2 + \delta_2} } \1_{Y_n \leq R_j(n)}] \\
&+ E^0_\omega \otimes E^0_\omega [Y^3_n \cdot K_n \1_{K_n \leq  n^{1/2 + \delta_2} } \1_{Y_n > R_j(n)}] \\
&+ E^0_\omega \otimes E^0_\omega [Y^3_n \cdot K_n \1_{K_n >  n^{1/2 + \delta_2} } \1_{Y_n \leq R_j(n)}] \\
&+ E^0_\omega \otimes E^0_\omega [Y^3_n \cdot K_n \1_{K_n >  n^{1/2 + \delta_2} } \1_{Y_n > R_j(n)}] \\
&\leq R^3_j(n)  n^{1/2+\delta} \\
&+ n^{1/2+\delta_2}E^0_\omega \otimes E^0_\omega [Y^3_n \1_{Y_n > R_j(n)}]\\
&+ R^3_j(n)E^0_\omega \otimes E^0_\omega [ K_n \1_{K_n >  n^{1/2 + \delta_2} } ]\\
&+E^0_\omega \otimes E^0_\omega [Y^6_n \1_{Y_n > R_j(n)}]^{1/2}\cdot E^0_\omega \otimes E^0_\omega [ K^2_n \1_{K_n >  n^{1/2 + \delta_2} } ]^{1/2},
\end{align*}
where the fourth term in the last inequality is bounded by Cauchy-Schwartz.
This implies that for $ \omega \in A_n \cap B_n$, we have for large enough $n$
\begin{align*}
E^0_\omega \otimes E^0_\omega [\mathcal I_n] \leq n^{1/2 + \delta_2+\epsilon},
\end{align*}
for arbitrary $\epsilon>0$. From \eqref{eq: int-bad-1} and \eqref{eq: int-bad-2}, we see that for all $\alpha < 1+ \frac{d-1}{3d}$
\begin{align*}
P((A_n \cup B_n)^c) \leq Ce^{-c(\log n)^{\alpha (j+2)/(j+3)}} + e^{-c_1n^{\delta_2 - \delta_1}} .
\end{align*}
If $\alpha'$ satisfy $  \alpha' <\alpha(j+2)/(j+3) < 1+ \frac{d-1}{3d}$, then
\begin{align*}
 P((A_n \cup B_n)^c) \leq  C e^{-c(\log n)^{\alpha'}}.
\end{align*} 
Since $j$ is arbitrary we conclude that the bound holds for any $\alpha' < 1+ \frac{d-1}{3d}$, which finishes the proof.
\end{proof}
\section{Quenched heat kernel estimates} \label{sec: heat-kernel}
In this section we state a semi-local limit theorem which will be used to prove Theorems \ref{thm: invariant}, \ref{thm: radon-quantitative} and \ref{thm: slowdown}. It extends \cite[Proposition 3.1]{berger2016local} in two ways. First, it applies to $d=2$ and $d=3$, which is due to the intersection estimates from Section \ref{sec: intersections}. Second, it provides quantitative bounds for the probability the environment has large quenched heat kernel, which we will need for proving Theorem \ref{thm: radon-quantitative}.
\begin{proposition}\label{prop: heatkernel}
Let $d\geq 2$. Assume $P$ is uniformly elliptic, i.i.d. and satisfies condition $(T)$. For every $\theta \in (0,1]$, let $F(N) = F(N,\theta)$ be the event that for every $z \in \tilde {\mathcal P}(0,N)$, every $(d-1)$-dimensional cube $\Delta \subset \partial^+\mathcal P(0,N)$ of side-length $N^\theta$ and every interval $I$ of length $N^\theta$
\begin{align*}
\vert P_\omega^z(X_{T_{\partial \mathcal P(0,N)}}\in \Delta,T_{\partial \mathcal P(0,N)} \in I )-\prob^z(X_{T_{\partial \mathcal P(0,N)}}\in \Delta,T_{\partial \mathcal P(0,N)} \in I ) \vert \leq \frac{CN^{\theta d}}{N^{d}}\frac{R_3(N)}{N^{\theta (d-1)/(d+2)}} .
\end{align*}
Then for any $\alpha < 1 + \frac{d-1}{3d}$, there exist constants $C,c>0$ such that $P(F(N)) > 1 -  C e^{-c(\log N)^{\alpha}}$ for all $N \in \N$.
\end{proposition}
\begin{remark}
The bound $C e^{-c(\log N)^{\alpha}}$ for every $\alpha < 1 + \frac{d-1}{3d}$ is sub-optimal. After proving Theorem \ref{thm: slowdown}, we will essentially prove that this bound holds for every $\alpha<d$, see Lemma \ref{lem: d-box}.
\end{remark}
We leave the proof of the proposition to Appendix \ref{sec: heatkernel-appendix}. We can now prove Theorem \ref{thm: invariant}.
\begin{proof}[Proof of Theorem \ref{thm: invariant} assuming Proposition \ref{prop: heatkernel}]
The existence of the invariant measure and the local CLT were proven in \cite[Theorems 1.10, 1.11]{berger2016local} for $d \geq 4$. The only part of their proof that depends on $d$ is the quenched heat kernel bounds in \cite[Proposition 3.1]{berger2016local}. Since Proposition \ref{prop: heatkernel} provides the corresponding estimate for $d=2$ and $d=3$, we are done.
\end{proof} 
To prove Theorem \ref{thm: slowdown}, we will need a stronger version of Proposition \ref{prop: intersection}. The next lemma uses the heat kernel bounds of Proposition \ref{prop: heatkernel} to estimate the number of intersections.
 \begin{lemma} \label{lem: intersections-upgrade}
Let $d \geq 3$ and assume $P$ is uniformly elliptic, i.i.d. and satisfies condition $(T)$. Then for each $ \epsilon >0$,
\begin{align*}
P\left( E_\omega^0 \otimes E_\omega^0[\mathcal I_n] > n^{\epsilon} \right) = n^{-\xi(1)}.
\end{align*}
 \end{lemma}
 \begin{proof}
Recall that $X^{(1)}$ and $X^{(2)}$ are independent random walks in the same environment. Before we bound the number of their intersections, we will condition the environment to have good properties. Fix $\delta \in (0,1)$, and define the sets
\begin{align*}
\mathcal A =\bigcup_{m=\lfloor n^\delta \rfloor}^n B\left(\E^0[X_{m}],R_5(m) \sqrt{m} \right), \qquad \mathcal B_\delta =  \{z \in \Z^d: \norm{z}_1 < \lfloor n^\delta \rfloor \},
\end{align*}
and let $\mathcal A_\delta = \mathcal A \cup \mathcal B_\delta$.
Note that since this is a nearest neighbor random walk, starting from the origin it will take at least $\lfloor n^\delta \rfloor$ steps to visit $\mathcal A_\delta^c$. From the union bound and Lemma \ref{lem: exit-ball}, we have
\begin{align}\label{eq: int-space}
\begin{split}
\prob^0(\exists m \leq n, \: X_m \in \mathcal A_\delta^c )&= \prob^0(\exists m \in [\lfloor n^\delta \rfloor, n], \: X_m \in \mathcal A_\delta^c )\\
&\leq  \sum_{m = \lfloor n^{\delta} \rfloor }^n\prob^0(X_m \not \in B(\E^0[X_m],R_5(m)m^{1/2}))\\
&\leq \sum_{m = \lfloor n^{\delta} \rfloor }^n Ce^{-c R_5(m)} \\
&\leq n\cdot  Ce^{-c R_5(\lfloor n^\delta \rfloor)} \\
&\leq \tilde Ce^{-\tilde c R_5(n)}.
\end{split}
\end{align}
Define the event
\begin{align*}
E_n = &\{ \omega \in \Omega : P_\omega^0(\exists m < n ,  \: X_m \in \mathcal A_\delta^c ) \leq   C e^{-c R_5(n)} \};
\end{align*}
by Markov's inequality and \eqref{eq: int-space}, we have $P(E_n^c)  =  n^{-\xi(1)}$. 

The number of intersection points of $X^{(1)}$ and $X^{(2)}$ in $\mathcal B_\delta$ is trivially bounded by $C n^{\delta d}$. Hence for $\omega \in E_n $, we have
\begin{align*}
E^0_\omega \otimes E^0_\omega [\mathcal I_n] &\leq \sum_{x \in \Z^d} P^0_\omega(X^{(1)} \: \text{visits} \: x)P^0_\omega(X^{(2)} \: \text{visits} \: x) \\
&= \sum_{x \in \Z^d} P_\omega^0(X \: \text{visits} \: x)^2\\
&\leq \sum_{x \in \mathcal B_\delta} P^0_\omega(X \: \text{visits} \: x)^2+ \sum_{x \in \mathcal A \cap \mathcal B^c_\delta} P^0_\omega(X \: \text{visits} \: x)^2+\sum_{x \in \mathcal A_\delta^c} P^0_\omega(X \: \text{visits} \: x)^2\\
&\leq C n^{\delta d}+\sum_{ x \in \mathcal A \cap \mathcal B^c_\delta} P^0_\omega(X \: \text{visits} \: x)^2+n^{-\xi(1)}.
\end{align*}
To bound the remaining number of intersection points, we will use Proposition \ref{prop: heatkernel}. We write
\begin{align*}
\mathcal A \cap \mathcal B^c_\delta= \bigcup_{m =\lfloor n^{\delta} \rfloor  }^{n} B(\E^0[X_{m}], R_5(m)m^{1/2}) \cap \mathcal B^c_\delta \cap H_{m} 
\eqqcolon  \bigcup_{m =\lfloor n^{\delta} \rfloor  }^{n} D_m.
\end{align*}
Fix $\epsilon>0$, and let $F_n$ be the event that for all $m > \lfloor n^{\delta} \rfloor$, for all $(d-1)$-dimensional boxes $\Delta$ of side-length $\lfloor m^{\epsilon} \rfloor $, we have $ P_\omega^0(X_{T_m} \in \Delta) \leq C m^{\epsilon(d-1)}/m^{(d-1)/2}$. By Proposition \ref{prop: heatkernel}, Lemma \ref{lem: ann-kernel2} and the union bound, $P(F_n^c) = n^{-\xi(1)}$. 

Let $x \in \mathcal D_m$ and let $\Delta$ be the $(d-1)$-dimensional box of side-length $\lfloor m^{\epsilon} \rfloor$ whose center is $x$. We have
\begin{align*}
P^0_\omega(X \: \text{visits} \: x) \leq P^0_\omega(X_{T_m} \in \Delta) + P^0_\omega(X \: \text{visits} \: x, X_{T_m} \not \in \Delta).
\end{align*}
From Lemma \ref{lem: sub-slowdown} and Theorem \ref{thm: regen} and the union bound, we have for $d \geq 2$
\begin{align*}
\prob^0(\exists \: 1 \leq k \leq n : \tau_{k}-\tau_{k-1}> R_6(n)) \leq n \cdot Ce^{-c (\log R_6(n))^{8/7}} \leq  Ce^{-c (\log n)^{64/63}}.
\end{align*}
Define the event 
\begin{align*}
G_n = \{\omega \in \Omega: P_{\omega}^0(\exists \: 1 \leq k \leq n :  \tau_{k}-\tau_{k-1}> R_6(n)) \leq Ce^{-c (\log n)^{64/63}} \}.
\end{align*}
By the last bound and Markov's inequality, we have $P(G_n^c)  = n^{-\xi(1)}$.  We now observe that 
\begin{align*}
P^0_\omega(X \: \text{visits} \: x,X_{T_m} \not \in \Delta) \leq P_{\omega}^0(\exists \: 1 \leq k \leq n :  \tau_{k}-\tau_{k-1}> R_6(n)).
\end{align*}
This is because for the random walk to visit $x \in H_m$ after $T_m$, the random walk must regenerate after at least $R_6(n)$ steps since $R_6(n) \ll m^{\epsilon}$ for all $m>n^{\delta}$. Hence we have that under the event $F_n \cap G_n$, 
\begin{align*}
P^0_\omega(X \: \text{visits} \: x)  \leq C m^{\epsilon (d-1)}/m^{(d-1)/2} + Ce^{-c (\log n)^{64/63}},
\end{align*}
for every $x \in \mathcal A \cap \mathcal B^c_\delta$.

Thus, for $\omega \in F_n \cap G_n$, we have
\begin{align*}
 \sum_{ x \in \mathcal A \cap \mathcal B^c_\delta}P^0_\omega(X \: \text{visits} \: x)^2    = & \sum_{m= \lfloor n^{\delta} \rfloor }^n\sum_{x \in D_m} P^0_\omega(X \: \text{visits} \: x)^2  \\
 \leq & C \sum_{m=\lfloor n^{\delta} \rfloor }^n \vert D_m \vert \cdot \frac{m^{2 \epsilon(d-1)}}{m^{d-1}}  \\
 \leq & C \sum_{m=\lfloor n^{\delta} \rfloor }^n R_5^{d-1}(m)m^{(d-1)/2} \frac{m^{2 \epsilon(d-1)}}{m^{d-1}}  \\
 \leq & C\sum_{m=\lfloor n^{\delta} \rfloor }^n R_6(m) \frac{m^{2 \epsilon(d-1)}}{m^{(d-1)/2}}  \\
 \leq & Cn^{c\delta}
 \end{align*}
where the last inequality used the fact that $d \geq 3$ and $\epsilon$ is chosen small enough. We have shown that under the events $ E_n \cap F_n \cap G_n$, we have $E^0_\omega \otimes E^0_\omega [\mathcal I_{n}] < C n^{\delta d }+ Cn^{c \delta} + n^{-\xi(1)}$. Since $\delta$ is arbitrary, we are done.
\end{proof}
With this result, the proof of Theorem \ref{thm: slowdown} relies on reviewing \cite{berger2012slowdown} and replacing their intersection bounds with Lemma \ref{lem: intersections-upgrade}, which we leave to Appendix \ref{sec: slowdown-appendix}.
\section{Tail estimates for $dQ/dP$} \label{sec: radon-quantitative}
The purpose of this section is to prove Theorem \ref{thm: radon-quantitative}. Let $Q_n$ be the law of $\overline{\omega}_n$ and define 
\begin{align} \label{eq: radon-approx}
f_n(\omega) = \sum_{z \in \Z^d}P^z_\omega(X_n =0).
\end{align}
We observe that $dQ_n = f_n dP$, since for measurable $A \subset \Omega$
\begin{align*}
Q_n(A) = E \left[\sum_{z \in \Z^d} P_\omega^0(X_n = z) \1_{\sigma_z \omega \in A} \right]=  E\left[\sum_{z \in \Z^d} P_\omega^z(X_n = 0) \1_{\omega \in A} \right] =E[f_n (\omega) \1_{\omega \in A}],
\end{align*}
where the first equality is by definition of $\overline \omega_n$, and the second equality is by translation invariance of $P$. To prove the theorem, we will prove tail estimates for $f_n$ and then show it approximates $dQ/dP$.
\begin{proposition}\label{prop: radon-approx}
Assume $P$ is uniformly elliptic, i.i.d. and satisfies condition $(T)$. For every $d \geq 3$ and $\alpha < d$, there exist constants $C,c>0$ such that for every $n \in \N$ and $u>1$,  
\begin{align*}
 P\left(f_n>  u \right)\leq Ce^{-c(\log u)^\alpha} + Ce^{-c (\log n)^2}.
\end{align*}
 Furthermore, if $d \geq 2$ and $P$ is nestling, then there exist constants $C,c>0$ such that for every $n \in \N$ and $u>1$, 
\begin{align*}
  P\left(f_n>  u \right)\geq  Ce^{-c (\log u)^d}.
\end{align*}
\end{proposition}
We will need the following theorem to prove $Q_n$ converges to $Q$.
\begin{theorem}[Kozlov \cite{kozlov1985method}] \label{thm: kozlov}
Assume $P$ is uniformly elliptic and i.i.d. Assume there exists an invariant probability measure $Q$ with respect to $\overline \omega_n$ such that $Q \ll P$. Then the following hold:
\begin{enumerate}
\item $Q$ is equivalent to $P$.
\item $\{\overline \omega_n: n \in \N_0\}$ with initial law $Q$ is ergodic.
\item $Q$ is the unique invariant probability measure for $\{\overline \omega_n: n \in \N_0\}$ which is absolutely continuous with respect to $P$.
\end{enumerate}
\end{theorem}
 \begin{proof}[Proof of Theorem \ref{thm: radon-quantitative} assuming Proposition \ref{prop: radon-approx}]
 In \cite[Theorem 3.1]{sznitman1999law}, Sznitman and Zerner proved that under Kalikow's condition, $Q_n$ converges weakly to some invariant distribution. Their proof relies on Kalikow's condition in order to assume finite moments for the regeneration times. Since this is implied by condition $(T)$ (for example Lemma \ref{lem: sub-slowdown}), we can rerun their same proof to conclude that $Q_n \to Q_\infty$ weakly, where $Q_\infty$ is an invariant distribution. By Theorem \ref{thm: kozlov}, to show that $Q_{\infty}$ is the unique invariant probability measure $Q$ which is equivalent to $P$, it is enough to show $Q_\infty \ll P$. Let $A \subset \Omega$ be any measurable set. Applying the Cauchy-Schwartz inequality, we have
\begin{align} \label{eq: cs}
Q_n(A)  = E[f_n(\omega)\1_{\omega \in A}]\leq (E f_n^2)^{1/2}P(A)^{1/2}.
\end{align}
From the tail estimates of Proposition \ref{prop: radon-approx}, we claim $\sup_{n \in \N} E[f_n^2] <  \infty$. Indeed, by the definition of $f_n$, and that $X_n$ is nearest-neighbor random walk, we have $f_n < Cn^{d}$ for $P$-a.e. $\omega$. We thus have
\begin{align*}
E[f_n^2] = \int_0^{ Cn^{2d} }P[f_n^2 >u] du \leq 1+ \int_1^{ Cn^{2d} } C'e^{-c(\log u)^\alpha} du+ C''n^{2d} e^{-c(\log n)^2}  <\infty. 
\end{align*}
Hence we can take limits on both sides of \eqref{eq: cs} and conclude that
\begin{align*}
Q_\infty(A) < C P(A)^{1/2},
\end{align*}
which implies that $Q_\infty \ll P$. In particular, we have $Q_n \xrightarrow{d} Q_\infty = Q$. Since $\{f_n \}_{n \in \N}$ is uniformly integrable, by the Dunford-Pettis Theorem (see \cite[Lemma 4.13]{kallenberg1997foundations}), for any $g \in L^1(P)$, we have
\begin{align*}
\lim_{n \to \infty} E[f_{n}(\omega)g(\omega)]= E[dQ/dP (\omega)g(\omega)].
\end{align*}
We claim that this implies $f_n$ converges to $dQ/dP$ in probability, which will conclude the proof of the theorem. By contradiction, if $f_n$ does not converge to $dQ/dP$ in probability, then there exist $\epsilon, \delta>0$ and a subsequence $\{n_k \}_{k \in \N}$ such that 
\begin{align*}
\forall k \in \N, \: P(E_{n_k} )  \coloneqq  P( \omega \in \Omega : f_{n_k}(\omega) - dQ/dP(\omega ) > \epsilon ) > \delta.
\end{align*}
Define the event $E = \limsup_{k \to \infty} E_{n_k} $, and observe that 
\begin{align*}
P(E) = \lim_{k \to \infty} P\left(\bigcup_{j=k}^\infty E_{n_j} \right) > \delta.
\end{align*}
For $\omega \in E$, there exists a sub-subsequence $n_{k_m}$ such that $f_{n_{k_m}}(\omega) - dQ/dP(\omega) > \epsilon $. By Fatou's lemma, we have
\begin{align*}
\liminf_{m \to \infty} E[(f_{n_{k_m}}(\omega) - dQ/dP(\omega)) \1_{\omega \in E}] \geq \epsilon \cdot P(E)> \epsilon \cdot \delta > 0,
\end{align*}
which is a contradiction.
\end{proof}
\subsection{Upper bound}
This section is dedicated to proving the upper bound in Proposition \ref{prop: radon-approx}. We will need a version of Proposition \ref{prop: heatkernel} which holds for $d$-dimensional box, which is the subject of Lemma \ref{lem: d-box}. We then extend this lemma to boxes whose size do not depend on $n$ in Lemma \ref{lem: smallbox-upper}.
 \begin{lemma} \label{lem: d-box}
Let $d\geq 3$, and assume $P$ is uniformly elliptic, i.i.d. and satisfies condition $(T)$. Fix $\theta \in (0,1/2]$, and let $E(n)=E(n,\theta)$ be the event that for every $z \in \tilde{\mathcal P}(0,n^{1/2})$ and every $d$-dimensional box $\Delta$ with side-length $L = \lfloor n^\theta \rfloor$, we have
\begin{align*}
\left \vert P_\omega^z(X_n \in \Delta) - \prob^z(X_n \in \Delta)\right \vert < C \frac{L^d}{n^{d/2}} \frac{R_3(L)}{L^{1/4}}.
\end{align*}
Then for any $\alpha < d$, there exist constants $C,c>0$ such that $P(E(n)) > 1 -  C e^{-c(\log n)^{\alpha}}$ for all $n \in \N$.
\end{lemma}
\begin{remark}
Aside from holding for $d$-dimensional boxes, this lemma also provides nearly sharp bounds on the tail probability the heat kernel is large. This is possible for $d \geq 3$ by the tail bounds for the regeneration times in \eqref{eq: berger-slowdown} and Theorem \ref{thm: slowdown}.
\end{remark}
We leave the proof of Lemma \ref{lem: d-box} to Appendix \ref{sec: d-box-appendix}.
\begin{lemma}\label{lem: smallbox-upper}
Let $d \geq 3$ and assume $P$ is uniformly elliptic, i.i.d. and satisfies condition $(T)$. Fix $n \in \N$ and let $L \in (0,n^{1/2}] \cap \N$. Let $x \in \Z^d$ and let $\Delta$ be a $d$-dimensional box with side-length $L$. For some $c'>0$, let $ G(\Delta)$ be the event that for every $z \in B(x,c'n^{1/2})$ we have
\begin{align*}
\vert P_\omega^z(X_{n} \in \Delta)- \prob^z(X_{n} \in \Delta) \vert \leq C\frac{L^d}{n^{d/2}}\frac{R_3(L)}{L^{1/4}}.
\end{align*}
Then for any $\alpha < d$, there exist constants $C,c,\tilde{c}>0$ such that $P(G(\Delta)^c) \leq C e^{-c(\log L)^\alpha}$ for any $n \in \N$ and $L \in (0,n^{1/2}] \cap \N$.
\end{lemma}
We will prove this lemma by induction on the size of the boxes. We will fix scales $n_k = \sum_{j=0}^k N_j$ for $N_j= n^{1/2^j}$, and bound $\lambda_k = \vert  P_\omega^z(X_{n_k} \in \Delta) - \prob^z(X_{n_k} \in \Delta)\vert $, where $\Delta$ is a box of side-length $N_k^{\theta}$ for some fixed $\theta>0$. Our scheme is to condition on time $n_{k-1}$, apply the induction hypothesis to bound $\lambda_{k-1}$, and then bound the remaining term with Lemma \ref{lem: d-box}. We iterate this step $r(n)$ times, for $r(n)$ satisfying $N_{r(n)}^{\theta} \leq L$.
\begin{proof}
By translation invariance of $P$, we will only consider the case when $x$ is the origin. The proof will use descending induction on the size of the boxes. Fix $\theta \in (0,1/2)$. For $j \in \N$, let $N_j = \lfloor n^{1/2^j} \rfloor $ and let $r(n) = \lfloor \log_2 (\frac{ \log n }{\theta \log L}) \rfloor $, which is the minimal natural number satisfying $N^\theta_{r(n)} \leq L$. Define $n_0 = n -\sum_{j=1}^{r(n)} N_j$ and $n_k = n_0 + \sum_{j=1}^k N_j$, so that $n_{r(n)} = n$. For $1 \leq k \leq r(n)$, let $\Pi_k$ be a partition of $\Z^d$ into $d$-dimensional boxes of side-length $N_k^\theta$. For a box $\Delta \in \Pi_k$, define  
\begin{align*}
\lambda(z, \Delta, n_k) &= \vert  P_\omega^z(X_{n_k} \in \Delta) - \prob^z(X_{n_k} \in \Delta)\vert , \\
M(N_k) &=  C\frac{N_{k}^{\theta d}}{n^{d/2}}\frac{R_3(N_k)}{N_k^{\theta/4} },
\end{align*}
and the induction hypothesis $\text{HK}(n_k)$:
\begin{align*} 
P(\forall z \in B(0,\tilde c n_{k}^{1/2}), \: \lambda(z,\Delta,n_{k}) > M(N_{k})) \leq  C e^{-c(\log N^\theta_{k})^\alpha}.
\end{align*}
We will show $\text{HK}(n_{k-1})$ implies $\text{HK}(n_k)$, which by induction and our choice of parameters will prove the lemma. 
The base case $k=0$ follows from Lemma \ref{lem: d-box} since $n_0 = O(n)$. 
 
 We now prove the inductive step. For fixed $\Delta \in \Pi_k$, we have
\begin{align*}
&\vert
P_{\omega}^{z} (X_{n_k} \in\Delta )-
\mathbb{P}^{z} (X_{n_k} \in\Delta)\vert
\nonumber
\\
 \leq&\sum_{\Delta' \in \Pi_{k-1}} 
\left\vert \sum_{y \in \Delta'} [ P_{\omega}^{z} (X_{n_{k-1}} =y, X_{n_k}\in\Delta ) -\mathbb{P}^{z}(X_{n_{k-1}} =y, X_{n_k}\in\Delta)] \right\vert \\
\leq &\sum_{\Delta' \in \Pi_{k-1}} \left \vert \sum_{y \in \Delta'} P_{\omega}^{z}(X_{n_{k-1}}=y)  \left[ P_{\omega}^{y}(X_{n_k - n_{k-1}} \in \Delta) - \prob^y(X_{n_k - n_{k-1}} \in \Delta) \right ] \right \vert \\
+&\sum_{\Delta' \in \Pi_{k-1}}  \left \vert \sum_{y \in \Delta'} \left[ P_{\omega}^{z}(X_{n_{k-1}}=y)  - \prob^{z}(X_{n_{k-1}}=y)  \right ] \prob^y(X_{n_{k}-n_{k-1}} \in \Delta )\right \vert\\
+&\sum_{\Delta' \in \Pi_{k-1}} \left \vert \sum_{y \in \Delta'}  \prob^{z}(X_{n_{k-1}}=y) \left[ \prob^y(X_{n_k - n_{k-1}} \in \Delta) - \prob^z(X_{n_k - n_{k-1}} \in \Delta \vert X_{n_{k-1}}=y )  \right ] \right \vert\\
  \eqqcolon& S_1 + S_2 + S_3.
\end{align*}
We now show that under the assumption $\text{HK}(n_{k-1})$, $S_1, S_2$ and $S_3$ are each bounded with high probability. We first bound $S_1$:
\begin{align*}
S_1 &= \sum_{\Delta' \in \Pi_{k-1}} \left \vert \sum_{y \in \Delta'} P_{\omega}^{z}(X_{n_{k-1}}=y)  \left[ P_{\omega}^{y}(X_{N_k} \in \Delta) - \prob^y(X_{N_k} \in \Delta) \right ] \right \vert \\
&= \sum_{\substack{\Delta' \in \Pi_{k-1}\\\text{dist}(\Delta', \Delta)\leq N_k}} \left \vert \sum_{y \in \Delta'} P_{\omega}^{z}(X_{n_{k-1}}=y)  \left[ P_{\omega}^{y}(X_{N_k} \in \Delta) - \prob^y(X_{N_k} \in \Delta) \right ] \right \vert \\
&\leq \sum_{\substack{\Delta' \in \Pi_{k-1}\\\text{dist}(\Delta', \Delta)\leq N_k}}  P_{\omega}^{z}(X_{n_{k-1}}\in \Delta') \max_{y \in \Delta'} \left \vert P_{\omega}^{y}(X_{N_k} \in \Delta) - \prob^y(X_{N_k} \in \Delta) \right \vert,
\end{align*} 
where the second equality follows from the random walk being nearest-neighbor. Define the event
\begin{equation*}
G_{k-1}  = G_{k-1}(\Delta) =  \left\{\omega \in \Omega: \begin{alignedat}{2}  \begin{gathered} \forall z \in B(0,cn^{1/2}_{k-1}),  \forall \Delta' \in \Pi_{k-1} \: \text{s.t.} \: \text{dist}(\Delta, \Delta') \leq N_k,  \\[.25cm]
\lambda(z, \Delta', n_{k-1}) < M(N_{k-1})  \end{gathered} \end{alignedat} \right \}.
\end{equation*}
Our hypothesis $\text{HK}(n_{k-1})$ and a union bound gives us
\begin{align*}
P(G_{k-1}^c) <  C\frac{N_k^d}{N^{\theta d}_k} e^{-c(\log N^{\theta}_{k-1})^{\alpha}} \leq  C' e^{-c'(\log N^{\theta}_{k-1})^{\alpha}}.
\end{align*}
For $\omega \in G_{k-1}$, we apply Lemma \ref{lem: ann-kernel2} to get 
\begin{align*}
P_{\omega}^{z}(X_{n_{k-1}}\in \Delta') &\leq\prob^{z}(X_{n_{k-1}}\in \Delta') + \vert P_{\omega}^{z}(X_{n_{k-1}}\in \Delta')  -\prob^{z}(X_{n_{k-1}}\in \Delta')  \vert \\
&\leq \frac{C N_{k-1}^{\theta d}}{n^{d/2}} + \frac{C N_{k-1}^{\theta d}}{n^{d/2}}\frac{R_3(N_{k-1})}{N^{\theta/4}_{k-1}} \\
&\leq \frac{C N_{k-1}^{\theta d}}{n^{d/2}}.
\end{align*}
Hence for $\omega \in G_{k-1}$ and $z \in B(0, c n_{k-1}^{1/2})$, we have
\begin{align*}
S_1 \leq \frac{C N_{k-1}^{\theta d}}{n^{d/2}} \sum_{\substack{\Delta' \in \Pi_{k-1}\\\text{dist}(\Delta', \Delta)\leq N_k}} \max_{y \in \Delta'} \left \vert P_{\omega}^{y}(X_{N_k} \in \Delta) - \prob^y(X_{N_k} \in \Delta) \right \vert.
\end{align*}
 We say $\Delta' \in \Pi_{k-1} $ is good if for every $y \in \Delta'$, we have for every $\Delta \in \Pi_k$
\begin{align} \label{eq: cond1}
\left \vert P_{\omega}^{y}(X_{N_k} \in \Delta) - \prob^y(X_{N_k} \in \Delta) \right \vert \leq \frac{C N_k^{\theta d}}{N_k^{d/2}}\frac{R_3(N_k)}{N_k^{\theta/4}},
\end{align}
and
\begin{align}\label{eq: cond2}
P_\omega^y \left( \norm{ X_{N_k} - \E^z[X_{N_k}]}_\infty  > R_{5}(N_k)N^{1/2}_k\right) &\leq C\exp(-c R_5(N_k)).
\end{align}
Else, we say $\Delta'$ is bad. By Lemma \ref{lem: exit-ball}, Lemma \ref{lem: d-box} and Markov's inequality, we have for any $\Delta'$
\begin{align*}  
P(\omega \in \Omega: \Delta'\: \: \text{is bad}) \leq  C\exp(-c (\log N_k)^\alpha).
\end{align*}
Define the event
\begin{align*}
A_k  = A_k(\Delta)=  \left\{ \omega \in \Omega: \forall \Delta' \in \Pi_{k-1} \: \text{s.t.} \: \text{dist}(\Delta, \Delta') \leq N_k, \Delta' \: \text{is good}  \right\}.
\end{align*}
From the union bound, $P(A_k^c) \leq C\exp(-c (\log N_k)^\alpha)$. For $\omega \in A_k$, we have by \eqref{eq: cond2} that if $\Delta'$ satisfies 
\begin{align*}
\max_{y \in \Delta'} \text{dist}(y + \E^y[X_{N_k}], \Delta) > R_5(N_k) N_k^{1/2}, 
\end{align*}
then
\begin{align*}
\max_{y \in \Delta'} \left \vert P_{\omega}^{y}(X_{N_k} \in \Delta) - \prob^y(X_{N_k} \in \Delta )\right \vert = N_k^{-\xi(1)}.
\end{align*}
The number of such boxes $\Delta'$ is at most
\begin{align*}
\frac{CR_5(N_k)N_k^{d/2}}{\vert \Delta' \vert} \leq \frac{CR_5(N_k) N_k^{d/2}}{N_{k-1}^{\theta d}}.
\end{align*}
Hence for $\omega \in A_k \cap G_{k-1}$ and $z \in B(0, c n_{k-1}^{1/2})$, we can apply \eqref{eq: cond1} and the previous estimate to get
\begin{align*}
S_1 &\leq  \frac{C N_{k-1}^{\theta d}}{n^{d/2}} \left( \frac{CR_5(N_k) N_k^{d/2}}{N_{k-1}^{\theta d}} \cdot \frac{C N_k^{\theta d}}{N_k^{d/2}}\frac{R_3(N_k)}{N_k^{\theta/4}} + N_k^{-\xi(1)} \right)\\
&\leq \frac{C N_k^{\theta d}}{n^{d/2}}\frac{R_3(N_k)}{N_k^{\theta /4}}.
\end{align*}
Next, we bound $S_2$:
\begin{align*}
S_2 & = \sum_{\Delta' \in \Pi_{k-1}}  \left \vert \sum_{y \in \Delta'}\prob^y(X_{N_k} \in \Delta ) \left[ P_{\omega}^{z}(X_{n_{k-1}}=y)  - \prob^{z}(X_{n_{k-1}}=y)  \right ] \right \vert\\
  & \leq \sum_{\Delta' \in \Pi_{k-1}}  \max_{y \in \Delta'} \prob^y(X_{N_k} \in \Delta) \left\vert P_{\omega}^{z}(X_{n_{k-1}} \in \Delta')  - \prob^{z}(X_{n_{k-1}} \in \Delta')  \right \vert\\
  & = \sum_{\substack{\Delta' \in \Pi_{k-1}\\ \text{dist}(\Delta, \Delta')\leq N_k}}  \max_{y \in \Delta'} \prob^y(X_{N_k} \in \Delta) \left\vert P_{\omega}^{z}(X_{n_{k-1}} \in \Delta')  - \prob^{z}(X_{n_{k-1}} \in \Delta')  \right \vert.
\end{align*}
For $\omega \in G_{k-1}$ and $z \in B(0, c n_{k-1}^{1/2})$, we have
\begin{align*}
S_2  &\leq \frac{C N_{k-1}^{\theta d}}{n^{d/2}}\frac{R_3(N_k)}{N^{\theta/4}_{k-1}}  \sum_{\substack{\Delta' \in \Pi_{k-1}\\ \text{dist}(\Delta, \Delta') \leq N_k}}  \max_{y \in \Delta'} \prob^y(X_{N_k} \in \Delta).
\end{align*}
Since we are dealing with the annealed measure, we can apply the same analysis for bounding $S_1$ and get
\begin{align*}
 &\sum_{\substack{\Delta' \in \Pi_{k-1}\\ \text{dist}(\Delta, \Delta') \leq N_k}}  \max_{y \in \Delta'} \prob^y(X_{N_k} \in \Delta) \\
 &\leq  \sum_{\substack{\Delta' \in \Pi_{k-1}\\ \max_{y \in \Delta'}\text{dist}(y + \E^y[X_{N_k}], \Delta) \leq R_5(N_k)N_k^{1/2}}}  \max_{y \in \Delta'} \prob^y(X_{N_k} \in \Delta) + N_k^{-\xi(1)} \\
 &\leq \frac{CR_5(N_k) N_k^{d/2}}{N_{k-1}^{\theta d}} \cdot \frac{C N_k^{\theta d}}{N_k^{d/2}}\\
 &\leq \frac{CR^d_5(N_k)}{N_k^{\theta d}}.
\end{align*}
Hence for $\omega \in G_{k-1}$ and $z \in B(0, c n_{k-1}^{1/2})$, we have
\begin{align*}
S_2 \leq \frac{CN_{k-1}^{\theta d}}{n^{d/2}}  \frac{R_3(N_{k-1})}{N_{k-1}^{\theta/4}} \cdot \frac{CR^d_5(N_k)}{N_k^{\theta d}} \leq \frac{C N_k^{\theta d}}{n^{d/2}}\frac{R_6(N_{k-1})R^d_5(N_{k})}{N_{k-1}^{\theta /4}}.
\end{align*}
Finally, we bound $S_3$:
\begin{align*}
S_3 = &\sum_{\Delta' \in \Pi_{k-1}} \left \vert \sum_{y \in \Delta'}  \prob^{z}(X_{n_{k-1}}=y) \left[ \prob^y(X_{N_k} \in \Delta) - \prob^z(X_{N_k} \in \Delta \vert X_{n_{k-1}}=y )  \right ] \right \vert\\
   &\leq \sum_{\Delta' \in \Pi_{k-1}} \prob^z(X_{n_{k-1}} \in \Delta')\left \vert \max_{y \in \Delta'} \prob^y(X_{N_k} \in \Delta) - \min_{y \in \Delta'} \prob^z(X_{N_k} \in \Delta  \vert X_{n_{k-1}}=y )  \right \vert\\
   &\leq C \frac{N^{\theta d}_{k-1}}{n^{d/2}}\sum_{\Delta' \in \Pi_{k-1}} \left \vert \max_{y \in \Delta'} \prob^y(X_{N_k} \in \Delta) - \min_{y \in \Delta'} \prob^z(X_{N_k} \in \Delta  \vert X_{n_{k-1}}=y )  \right \vert\\
  &= C \frac{N^{\theta d}_{k-1}}{n^{d/2}}\sum_{\substack{\Delta' \in \Pi_{k-1}\\\text{dist}(\Delta,\Delta') \leq N_k}} \left \vert \max_{y \in \Delta'} \prob^y(X_{N_k} \in \Delta) - \min_{y \in \Delta'} \prob^z(X_{N_k} \in \Delta  \vert X_{n_{k-1}}=y )  \right \vert,
\end{align*}
where the second inequality follows from Lemma \ref{lem: ann-kernel1}. For any $x,y \in \Delta'$, we have from Lemma \ref{lem: ann-kernel1}
\begin{align*}
\left \vert  \prob^y(X_{n_{k}-n_{k-1}} \in \Delta) - \prob^z(X_{n_{k}} \in \Delta \vert X_{n_{k-1}}=x )  \right \vert &\leq \frac{\vert x - y \vert  \cdot C N_k^{\theta d}  }{N^{(d+1)/2}_{k}}\leq \frac{C N_k^{\theta d} N^{\theta}_{k-1}}{N^{(d+1)/2}_{k}}.
\end{align*}
From earlier analysis, we observe that the boxes $\Delta'$ satisfying $\max_{y \in \Delta'}\text{dist}(y + \E^y[X_{N_k}], \Delta) > R_5(N_k)N_k^{1/2}$ have a negligible contribution to the sum. We thus get for $P$-a.e. $\omega$
\begin{align*}
S_3 &\leq C\frac{N^{\theta d}_{k-1}}{n^{d/2}}\sum_{\substack{\Delta' \in \Pi_{k-1}\\ \max_{y \in \Delta'}\text{dist}(y + \E^y[X_{N_k}], \Delta) \leq R_5(N_k)N_k^{1/2}}} \frac{C N_k^{\theta d} N^{\theta}_{k-1}}{N^{(d+1)/2}_{k}} \\
 &\leq C\frac{N^{\theta d}_{k-1}}{n^{d/2}} \cdot \frac{R^d_5(N_{k}) N^{d/2}_{k}}{N^{\theta d}_{k-1}} \cdot \frac{C N_k^{\theta d} N^{\theta}_{k-1}}{N^{(d+1)/2}_{k}} \\
 &\leq C\frac{N_k^{\theta d}}{n^{d/2}}\frac{N^\theta_{k-1}}{N_k^{1/2}}.
\end{align*}
To summarize, we have shown that for $\omega \in A_k \cap G_{k-1}$ and $z \in B(0, c n_{k-1}^{1/2})$, we have
\begin{align*}
\lambda(z,\Delta, n_{k}) &\leq \frac{CN^{\theta d}_{k}}{n^{d/2}}\left(\frac{R_3(N_k)}{N_k^{\theta/4}} + \frac{R_3(N_{k-1})R^d_5(N_{k})}{N_{k-1}^{\theta/4}} + \frac{N^\theta_{k-1}}{N_k^{1/2}} \right) \\
 &\leq \frac{CN^{\theta d}_{k}}{n^{d/2}}\frac{R_5(N_k)}{N_k^{\theta /4}},
\end{align*}
where the last inequality is due to $N_{k-1} =  N^2_{k}$ and $\theta< 1/2$. Applying this step iteratively, we get
\begin{align*}
\bigcap_{k=1}^{r(n)} A_k \cap G_{k-1} \subset \{ \forall z \in B(0,cn^{1/2}), \: \lambda(z,\Delta,L) > M(L) \}.
\end{align*}
We are only left to bound the probability of this event:
\begin{align*}
P\left(\bigcup_{k=1}^{r(n)} A^c_k \cup G^c_{k-1}  \right) \leq\sum_{k=1}^{r(n)}(P(A^c_k)+ P(G^c_{k-1}))\leq \sum_{k=1}^{r(n)}2 e^{-c (\log N_k)^\alpha} \leq C e^{-c (\log L )^\alpha}
\end{align*}
for any $\alpha < d$, which finishes the proof.
\end{proof}
\begin{remark} \label{rem: box-switch}
We will later also need a bound for hitting a finite $(d-1)$-dimensional box in a finite interval. Let $\Delta$ be a $d$-dimensional box with side-length $L$ centered at $c(\Delta)$. Let $\tilde \Delta$ be a $(d-1)$-dimensional box with side-length $\lfloor L/10 \rfloor$, also centered at $c(\Delta)$, and let $\tilde I = [n-\lfloor L/10 \rfloor, n+\lfloor L/10 \rfloor]$. Since our random walk is nearest-neighbor, we have
\begin{align*}
P_\omega^z(X_{T} \in \tilde \Delta, T \in \tilde I ) = P_\omega^z(X_{T} \in \tilde \Delta, T \in \tilde I, X_n \in \Delta) \leq P_\omega^z(X_n \in \Delta).
\end{align*}
We thus have
\begin{align*}
&P\left( \forall z \in \tilde{\mathcal P}(x,n^{1/2}), P_\omega^z(X_{T} \in \tilde \Delta, T \in \tilde I ) >C \lfloor L/10 \rfloor^d /n^{d/2}  \right) \\
&\leq P\left( \forall z \in B(x,\tilde cn^{1/2}),  P_\omega^z(X_n \in \Delta) >C \lfloor L/10 \rfloor^d /n^{d/2}   \right) \\
&\leq P\left( \forall z \in B(x,\tilde cn^{1/2}),  \vert P_\omega^z(X_n \in \Delta)- \prob^z(X_n \in \Delta) \vert +\prob^z(X_n \in \Delta)   >C \lfloor L/10 \rfloor^d /n^{d/2}   \right) \\
&\leq P\left( \forall z \in B(x,\tilde cn^{1/2}),  \vert P_\omega^z(X_n \in \Delta)- \prob^z(X_n \in \Delta) \vert  >\tilde C L^d/n^{d/2}  \right) \\
&\leq Ce^{-c(\log L)^{\alpha}},
\end{align*}
where the third inequality follows from Lemma \ref{lem: ann-kernel1}. Hence we have an upper bound on $P_\omega^z(X_{T} \in \tilde \Delta, T \in \tilde I ) $ for typical environments.
\end{remark}
Before we prove the upper bound in Proposition \ref{prop: radon-approx}, we need one last auxiliary result. 
\begin{lemma} \label{lem: mdp}
Assume $P$ is uniformly elliptic, i.i.d. and satisfies condition $(T)$. For $d \geq 3$, there exists constants $c,C>0$ such that for $u > 1$ we have
\begin{align*}
\prob^z(\norm{X_n - \E^z[X_n]}_\infty > u \cdot n^{1/2}) \leq C e^{-cu} + Ce^{-c(\log n)^2}.
\end{align*}
\end{lemma}
\begin{proof}
Let $Z_k = X_{\tau_k} - X_{\tau_{k-1}}$, and define the event
\begin{align*}
A_n = \{\forall k \in [n], \norm{ \tau_k - \tau_{k-1}}_\infty < R_1(n)  \}.
\end{align*}
By Theorem \ref{thm: regen}, \eqref{eq: berger-slowdown} and Theorem \ref{thm: slowdown}, we have $\prob^z(A_n^c) \leq Ce^{-c(\log R_1(n))^\alpha}$ for every $\alpha<d$. Since $d \geq 3$, we can choose $\alpha$ arbitrarily close to $3$ such that
\begin{align*}
\prob^z(A_n^c)  \leq e^{-c(\log R_1(n))^\alpha} = e^{-c(\log n)^{3\alpha/4}} \leq e^{-c(\log n)^2} 
\end{align*} 
Letting $\gamma_n = \max \{k \geq 0: \tau_{k} \leq n\}$, we have
\begin{align*}
X_n - X_0 = \sum_{k=1}^{\gamma_n} Z_k + (X_n-X_{\tau_{\gamma_n}}),
\end{align*}
which implies
\begin{align*}
\prob^z(\norm{X_n - \E^z[X_n]}_\infty > u \cdot n^{1/2}) \leq \prob^z \left(\sum_{k=1}^{n} \norm{Z_k - \E^z[Z_k]}_\infty > u \cdot n^{1/2}/2 , A_n\right) + P(A_n^c),
\end{align*}
where we used the fact $\gamma_n \leq n$ and $\norm{X_n-X_{\tau_{\gamma_n}}}_\infty \leq R_1(n)$ on $A_n$. Applying Bernstein's inequality, see \cite[Section 2.8]{boucheron2013concentration}, and Theorem \ref{thm: regen}, we have for $u >1$
\begin{align*}
 \prob^z \left(\sum_{k=1}^{n} \norm{Z_k - \E^z[Z_k]}_\infty > u \cdot n^{1/2}/2 , A_n\right) \leq \exp\left(-\frac{cu^2 n }{ c n + c uR_1(n)n^{1/2}}  \right) \leq e^{-c'u}.
\end{align*}
Putting everything together gives us
\begin{align*}
\prob^z(\norm{X_n - \E^z[X_n]}_\infty > u \cdot n^{1/2}) \leq e^{-c'u} + Ce^{-c (\log n)^2}.
\end{align*}
\end{proof}
\begin{proof}[Proof of the upper bound in Proposition \ref{prop: radon-approx}]
Define $L$ to be the largest integer such that $L^d \leq u/4$. For any $\epsilon> 0$, we have
\begin{align*}
P\left( \sum_{z \in \Z^d}P_\omega^z(X_n = 0) > u\right) &\leq P\left( \sum_{z \in  B(-nv, L^\epsilon \cdot n^{1/2})} P_\omega^z(X_n = 0) > u/2  \right)  \\
& +  P\left( \sum_{z \not \in B(-nv, L^\epsilon \cdot n^{1/2})} P_\omega^z(X_n = 0) > u/2  \right)  .
\end{align*}
Let $\Delta$ be a $d$-dimensional box centered at the origin with side-length $L$. We bound the first term:
\begin{align*}
&P\left(  \sum_{z \in B(-nv, L^\epsilon \cdot n^{1/2})}P_\omega^z(X_n = 0) > u/2   \right) \\
&\leq P\left(  \sum_{z \in B(-nv,L^\epsilon \cdot n^{1/2})}P_\omega^z(X_n \in \Delta)  > u/2 \right) \\
&\leq P\left(  \sum_{z \in B(-nv,L^\epsilon \cdot n^{1/2})} \left( \prob^z(X_n \in \Delta) +  \vert \prob^z(X_n \in \Delta) - P_\omega^z(X_n \in \Delta)  \vert \right)  > u/2   \right) .
\end{align*}
Using translation invariance of $\prob$,
\begin{align*}
\sum_{z \in \Z^d}\prob^z(X_n \in \Delta) &= \sum_{z \in \Z^d}\prob^0(X_n \in \Delta-z) \\
&= \E^0\left[ \sum_{x \in \Z^d} \1_{\{X_n = x\}} \sum_{z \in \Z^d} \1_{\{x+ z \in \Delta\}}\right] \\
&= \vert \Delta \vert \cdot \E^0\left[ \sum_{x \in \Z^d} \1_{\{X_n = x\}} \right] \\
&= \vert \Delta \vert \\
& \leq u/4,
\end{align*}
where the last inequality follows from the definition of $L$ and $\Delta$. We then have
\begin{align*}
&P\left( \sum_{z \in  B(-nv, L^\epsilon \cdot n^{1/2})} P_\omega^z(X_n = 0) > u/2  \right)\\
& \leq P\left( \sum_{z \in  B(-nv, L^\epsilon \cdot n^{1/2})}\vert \prob^z(X_n \in \Delta) - P_\omega^z(X_n \in \Delta)  \vert > u/4 \right)\\
&\leq P\left(    \exists z \in B(-nv,L^\epsilon \cdot n^{1/2}), \: \: \vert \prob^z(X_n \in \Delta) - P_\omega^z(X_n \in \Delta)  \vert >C\frac{u}{n^{d/2}L^{\epsilon d}}   \right) \\
&\leq P\left(    \exists z \in B(-nv,L^\epsilon \cdot n^{1/2}), \: \: \vert \prob^z(X_n \in \Delta) - P_\omega^z(X_n \in \Delta)  \vert >C \frac{u}{n^{d/2}}\frac{R_3(L)}{L^{1/4}}   \right) \\
& \leq C\exp(-c(\log u)^\alpha),
\end{align*}
where the last inequality holds from Lemma \ref{lem: smallbox-upper} and the fact that $L \asymp u^{1/d}$. Note that the fourth inequality holds by choosing $\epsilon>0$ small enough. We bound the second term:
\begin{align*}
 P\left( \sum_{z \not \in B(-nv, L^\epsilon \cdot n^{1/2})} P_\omega^z(X_n = 0) > u/2  \right) &\leq \frac{2}{u} \sum_{z \not \in B(-nv, L^\epsilon \cdot n^{1/2})} \prob^z(X_n = 0) \\
&= \frac{2}{u} \sum_{z \not \in B(nv, L^\epsilon \cdot n^{1/2})} \prob^0(X_n = z) \\
&= \frac{2}{u}  \prob^0(X_n \not \in B(nv, L^\epsilon \cdot n^{1/2})),
\end{align*}
where the first inequality is Markov's inequality and the first equality is by translation invariance of $\prob$. Choosing $L$ large enough such that $L^{\epsilon}>1$, we can apply Lemma \ref{lem: mdp} and get
\begin{align*}
 \prob^0(X_n \not \in B(nv, L^\epsilon \cdot n^{1/2})) \leq C e^{-cL^{\epsilon}} + Ce^{-c (\log n)^2}.
\end{align*}
Putting everything together, we conclude that
\begin{align*}
P\left( \sum_{z \in \Z^d}P_\omega^z(X_n = 0) > u\right) & \leq C e^{-c(\log u)^{\alpha}} + C e^{-cL^{\epsilon}}+ Ce^{-c (\log n)^2}\\
& \leq C' e^{-c(\log u)^{\alpha}} + Ce^{-c (\log n)^2}.
\end{align*}

\end{proof}

\subsection{Lower bound}
The main result in this section is Lemma \ref{lem: smallbox-lower}, where we derive a lower bound for $P_\omega^z(X_n=0)$ by placing a trap at the origin. First, we will need a lower bound for the annealed heat kernel.
\begin{lemma} \label{lem: ann-lower}
Let $d \geq 2$ and assume $P$ is uniformly elliptic, i.i.d. and satisfies condition $(T)$.  There exist constants $C,c>0$ such that for $z \in \Z^d$ and $x \in H_{n + \langle z, e_1 \rangle}$ satisfying $\norm{x - \E^z[X_{T_{n+ \langle z, e_1 \rangle}}]}_\infty < C n^{1/2}$ and $m \in [\E^z[T_{n+ \langle z, e_1 \rangle}]- Cn^{1/2},\E^z[T_{n+ \langle z, e_1 \rangle}] + Cn^{1/2}]$, we have
\begin{align*}
\prob^z(X_{T_{n+ \langle z, e_1 \rangle}}  = x,T_{n+ \langle z, e_1 \rangle}= m) > c/n^{d/2}.
\end{align*}
\end{lemma}
\begin{proof}
The argument closely follows the proof of \cite[ Lemma 4.4]{berger2012slowdown}. By translation invariance of $P$, it is enough to prove the lower bound for the case when $z=0$. For $k,n \in \N$, denote 
\begin{align*}
B(n,k) = \{ \langle X_{\tau_k} , e_1 \rangle = n\}
\end{align*}
and $B(n) = \cup_{k \in \N} B(n,k)$. Suppose that $x \in H_{n}$ and $ m \in \N$ satisfy 
\begin{align}\label{eq: assume}
&\vert x - k \E^z[X_{\tau_2}-X_{\tau_1}] \vert < Ck^{1/2} \qquad \text{and} \qquad \vert m - k\E^z[\tau_2-\tau_1] \vert < C k^{1/2}
\end{align}
 for $k \in [M-M^{1/2}, M+M^{1/2}] $ for some $M$ which will be chosen later and satisfies $M = O(n)$. By Theorem \ref{thm: regen} and the local central limit theorem for sum of i.i.d. lattice-valued random variables, see \cite[Theorem 2.1.1]{lawler2010random}, under \eqref{eq: assume} we have
\begin{align*}
\prob^0(X_{T_{n}} = x,T_{n}= m, B(n,k) )=\prob^0(X_{\tau_k} = x,\tau_k= m, B(n,k) ) \geq c/k^{(d+1)/2}.
\end{align*}
This yields us
\begin{align*} 
\prob^0(X_{T_{n}} = x, T_{n} = m ) &\geq \prob^0(X_{T_{n}} = x, T_{n} = m,B(n) ) \\
&\geq \sum_{k =  M-M^{1/2} }^{M+M^{1/2}} \prob^0(X_{\tau_k} = x, \tau_k = m, B(n,k) )  \\
&\geq 2 c M^{1/2}/n^{(d+1)/2} \\
&=O (n^{-d/2}).
\end{align*}
Hence it is enough to show that if $x$ and $m$ satisfy the assumptions of the lemma, then they also satisfy \eqref{eq: assume}. Let $\beta_n = \inf \{j \in \N : \langle X_{\tau_{j+1}} -X_{\tau_1}, \ell \rangle \geq  n \}$, and note that
\begin{align*}
\{\beta_n = k\} = \left\{\sum_{j=1}^{k-1} \langle X_{\tau_{j+1}}-X_{\tau_{j}}, \ell \rangle    < n \leq \sum_{j=1}^{k} \langle X_{\tau_{j+1}}-X_{\tau_{j}}, \ell \rangle  \right \}.
\end{align*}
Thus $\beta_n$ is a stopping time for the filtration of the process $\{X_{\tau_{j+1}}-X_{\tau_{j}}: j \in \N \}$, and by Wald's identity we have 
\begin{align*}
\E^0[\beta_n ]\E^0[X_{\tau_2}-X_{\tau_1}]= \E^0 \left[\sum_{j=1}^{\beta_n}(X_{\tau_{j+1}}-X_{\tau_{j}}) \right].
\end{align*}
Hence we can write
\begin{align*}
  \E^0[X_{T_n}] =\E^0[ X_{T_n}-X_{\beta_n+1}] +\E^0[\beta_n ]\E^0[X_{\tau_2}-X_{\tau_1}].
\end{align*}
Let $M = \E^0[\beta_n ]$ and note that $M \asymp n$, see \cite[Lemma 5.1]{sznitman2000slowdown}. For $k \in  [M-M^{1/2},M+M^{1/2}]$, we have $\norm{x - \E^0[X_{T_n}]}_\infty < Cn^{1/2} = O(k^{1/2})$, which implies
\begin{align*}
\norm{x - k\E^0[X_{\tau_2} -X_{\tau_1}]}_\infty &\leq  \norm{ x - \E^0[X_{T_n}]}_\infty  + \norm{\E^0[X_{T_n}]- k\E^0[X_{\tau_2} -X_{\tau_1}] }_\infty \\
&\leq C k^{1/2} + \norm{  M \E^0[X_{\tau_2}-X_{\tau_1}]+\E^0[X_{T_n}-X_{\beta_n+1} ]-k\E^0[X_{\tau_2} -X_{\tau_1}] }_\infty \\
&\leq C k^{1/2} +  C M^{1/2} \norm{\E^0 [X_{\tau_2}-X_{\tau_1}]}_\infty \\
&= O(k^{1/2}).
\end{align*}
Applying a similar analysis, we have that if $m \in [\E^0[T_n]- Cn^{1/2},\E^0[T_n] + Cn^{1/2}]$, then $m \in [ k \E^0[\tau_2 - \tau_1]- Ck^{1/2},k \E^0[\tau_2 - \tau_1] + Ck^{1/2}]$ for $k \in [M-M^{1/2},M+M^{1/2}]$. This finishes the proof.
\end{proof}

\begin{lemma} \label{lem: smallbox-lower}
Let $d \geq 2$ and assume $P$ is uniformly elliptic, i.i.d., satisfies condition $(T)$ and is nestling. There exist constants $c,c',C>0$ such that for all $u>1$ and for large enough $n \in \N$, we have
\begin{align*}
P( \forall z \in B(-nv, c'n^{1/2}), \: \: P_\omega^z (X_n=0) > u/n^{d/2}) \geq C e^{-c (\log u)^d}.
\end{align*}
\end{lemma}
\begin{proof}
Fix $L \geq 1$, and define the parallelogram centered at the origin
\begin{align*}
\Delta = \left \{ x\in \Z^d: \vert \langle x , e_1 \rangle \vert < L, \norm{x  -\vartheta \cdot \frac{\langle x,e_1 \rangle}{\langle \vartheta, e_1 \rangle} }_\infty<L \right\}.
\end{align*}
To prove the lower bound, we will need to consider two events. The first is the atypical event that $\Delta$ is a trap. The second is the typical event that quenched random walk behaves like the annealed random walk up until the first time it reaches $\Delta$.
 
We recall the na\"{i}ve trap event from \cite{sznitman2000slowdown}
\begin{align*}
 E_L = \left \{\omega \in \Omega : \forall y \in \Delta \setminus \{ 0 \}, \: \langle d(y,\omega) , y/\vert y \vert \rangle \leq- c_{2} \right \},
\end{align*}
where $c_{2}>0$ is some constant, see Figure \ref{fig: trap}. We will need the following estimate.
\begin{figure}[t!]  \centering
    \includegraphics[scale=0.50]{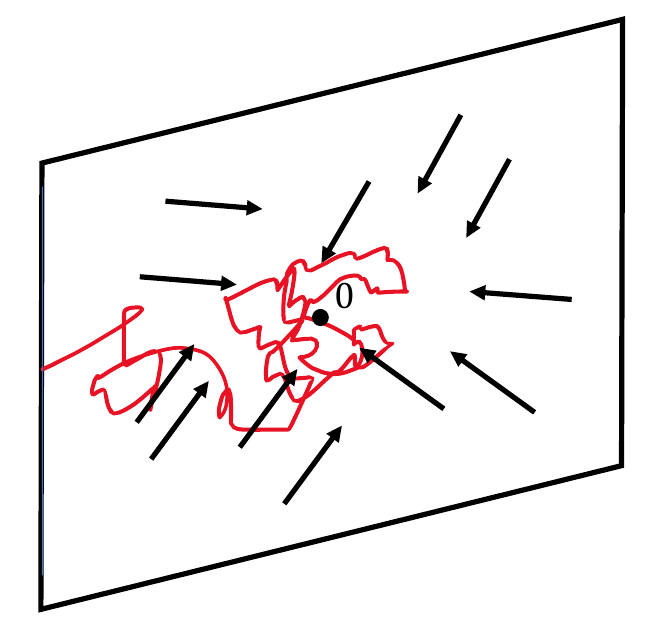}
    \caption{\textit{The na\"{i}ve trap is the environment where all the drifts in $\Delta$ are roughly pointing towards the origin. The random walk (in red) will typically spend a long time in this trap.}}
        \label{fig: trap}
\end{figure}
\begin{lemma}\label{lem: trap-lowerbound}
Suppose $\omega \in  E_L$. There exist positive constants $c_3, c_5, p_0$ such that for $y \in  \Delta \setminus B(0,c_3)$
\begin{align*}
P_\omega^y \left(  T_{\{0\}} > c_5 L    \right) < 1- p_0.
\end{align*}
\end{lemma}
\begin{proof}
 We first recall a few facts about the na\"{i}ve trap. From \cite[Lemma 2.8]{sznitman2000slowdown}, we see that there exist positive constants $c_3, c_4$, such that for $y \in  \Delta \setminus B_2(0,c_3)$ and $\omega \in E_L$, 
\begin{align*}
M_n = \exp(c_4 \vert X_{n \wedge T_{B_2(0,c_3)} \wedge T_{ \Delta^c}} \vert)
\end{align*} 
is a $P_\omega^y$-supermartingale. By uniform ellipticity, it is enough to prove the statement of the lemma with $T_{\{0\}}$ replaced by $T_{B_2(0,c_3)}$. We decompose this probability
\begin{align*}
P_\omega^y(T_{B_2(0,c_3)}> c_5 L) &\leq P_\omega^y(T_{\Delta^c} < T_{B_2(0,c_3)}) + P_\omega^y(T_{\Delta^c}> T_{B_2(0,c_3)}> c_5L) .
\end{align*}
By the optional stopping theorem applied to $M_n$, we obtain for any $y \in \Delta \setminus B_2(0,c_3)$ and $u>0$
\begin{align}\label{eq: trap-ineq}
P_\omega^y(  \vert X_{n \wedge T_{B_2(0,c_3)} \wedge T_{ \Delta^c}} \vert > u )\leq e^{-c_4(\vert y \vert - u)}.
\end{align}
Applying this inequality and uniform ellipticity, we obtain
\begin{align*} 
P_\omega^y(T_{ \Delta^c} < T_{B_2(0,c_3)}) \leq (1-p_0)/2
\end{align*}
for some $p_0\in (0,1)$. To bound the second summand, we note that from the proof of \cite[Lemma 2.8]{sznitman2000slowdown}, for $x \in \Delta \setminus B_2(0,c_3)$ we have
\begin{align*}
E_\omega^x[e^{\lambda \vert X_1 \vert}] \leq e^{\lambda \vert x \vert }\left(1 + \lambda \langle d(x,\omega), \frac{x}{\vert x \vert} \rangle  + O(\lambda/ \vert x \vert) + O(\lambda^2) \right).
\end{align*}
Since $\vert x \vert \geq c_3$, and by the definition of $E_L$, we can choose $\lambda$ small enough and $c_3$ large enough such that 
\begin{align*}
E_\omega^x[e^{\lambda \vert X_1 \vert}] \leq e^{\lambda \vert x \vert }\left(1 - c_2 \lambda/2  \right) \leq e^{\lambda \vert x \vert - c_2 \lambda /2} .
\end{align*}
From this estimate and the Markov property,
\begin{align*}
E_\omega^x[e^{\lambda \vert X_n  \vert}\1_{\{  T_{\Delta^c} \wedge T_{B_2(0,c_3)}  \geq n  \} }]&= E_\omega^x[ E_\omega^{X_{n-1}}[ e^{\lambda \vert X_n  \vert} \1_{\{  T_{\Delta^c} \wedge T_{B_2(0,c_3)} \geq n  \} } ] ] \\ 
&\leq e^{ - c_2 \lambda /2} E_\omega^x[ e^{\lambda \vert X_{n-1} \vert}  \1_{\{  T_{\Delta^c} \wedge T_{B_2(0,c_3)} \geq n  \} }],
\end{align*}
and continuing iteratively, we get
\begin{align*}
E_\omega^x[e^{\lambda \vert X_n  \vert}\1_{\{  T_{\Delta^c} \wedge T_{B_2(0,c_3)}\geq n\}  }  ] \leq e^{ - c_2 \lambda n /2 + \lambda \vert x \vert}.
\end{align*}
Applying Chebyshev's inequality and setting $n = c_5L$, we obtain
\begin{align*}
P_\omega^y(T_{\Delta^c}> T_{B_2(0,c_3)}> c_5 L) \leq P_\omega^y(\vert X_{c_5 L } \vert > c_3,T_{\Delta^c} \wedge T_{B_2(0,c_3)}  \geq c_5 L ) \leq \exp( -c_2 c_5 \lambda  L /2 + \lambda L).
\end{align*}
Taking $c_5$ large enough finishes the proof.
\end{proof}

Next, we define the typical event.  Fix $\epsilon>0$. Define $\partial^- \Delta =  \Delta \cap H_{-L}$, and partition it into $(d-1)$-dimensional boxes $\tilde \Delta$ of side-length $\lfloor L^\epsilon \rfloor$. Let $c$ be as in Lemma \ref{lem: ann-lower}, and let $I = [(n - L)/\langle v,e_1 \rangle - cn^{1/2},  (n-L)/\langle v,e_1 \rangle +cn^{1/2}]$ and partition it into intervals $\tilde I$ of side-length $\lfloor L^\epsilon \rfloor$. From Remark \ref{rem: box-switch}, we have 
\begin{align*}
P(\exists z \in \tilde {\mathcal P}(-nv, n^{1/2}), \:  P_\omega^z(X_{T_{-L}} \in \tilde \Delta , T_{-L} \in \tilde I) > C L^{\epsilon d}/n^{d/2})  = L^{-\xi(1)}.
\end{align*}
 Define the event 
\begin{equation*}
 F = F(n,L, \epsilon)  = \left\{\omega \in \Omega : \begin{alignedat}{2} \begin{gathered}  \forall z \in B(-nv, cn^{1/2}), \forall \tilde \Delta \subset \partial^- \Delta , \forall \tilde I \subset I, \\[.25cm]   P_\omega^z(X_{T_{-L}} \in \tilde \Delta , T_{-L} \in \tilde I)  \leq \frac{C L^{\epsilon d}}{n^{d/2}}\end{gathered}\end{alignedat} \right \}.
\end{equation*}
Choosing $c$ small enough, we have $B(-nv,c n^{1/2}) \subset \tilde{\mathcal P}(-nv, n^{1/2})$. Hence from the union bound we have $P(F) = 1- L^{-\xi(1)}$. 

We will now prove a lower bound for $P^z_\omega(X_n= 0)$ when $\omega \in  F \cap  E_L $. We have
\begin{align*}
P_\omega^z(X_n = 0) &= \sum_{y \in H_{-L} }\sum_{m < n}P_\omega^z(X_{T_{\{0\}}} = y,T_{\{0\}}=m ) P_\omega^y(X_{n-m} = 0) \\
 &\geq \sum_{y \in  \partial^- \Delta }\sum_{m \in I}P_\omega^z(X_{T_{-L}} = y,T_{-L}=m ) P_\omega^y(X_{n-m} = 0) \\
 &= \sum_{ \tilde \Delta \subset \partial^- \Delta} \sum_{\tilde I \subset I} \sum_{y \in \tilde \Delta}\sum_{m \in \tilde I } P_\omega^z(X_{T_{-L}} = y,T_{-L}=m ) P_\omega^y(X_{n-m} = 0).
\end{align*}
We force the random walk to hit the origin
\begin{align*}
P_\omega^y(X_{n-m} = 0) &> P_\omega^y(X_{n-m}=0, T_{\{0\}}< c L) \\
 &=\sum_{k=L}^{\lfloor cL \rfloor } P_\omega^y(T_{\{0\}}=k)P_\omega^0(X_{n-m-k}=0)  \\
  &\geq p_0 \min_{k \in[L, c L  ]} P_\omega^0(X_{n-m-k}=0),
\end{align*}
where the last lower bound is from Lemma \ref{lem: trap-lowerbound}. From this inequality, and that $\omega \in  F \cap  E_L $, we have
\begin{align*}
P_\omega^z(X_n = 0) &\geq C \sum_{\tilde \Delta \subset \partial^- \Delta} \sum_{\tilde I \subset I} \sum_{y \in \tilde \Delta}\sum_{m \in \tilde I } P_\omega^z(X_{T_{-L}} =y,T_{-L}=m )\min_{k \in[L, c L  ]} P_\omega^0(X_{n-m-k}=0)\\
&\geq C \sum_{\tilde \Delta \subset \partial^- \Delta} \sum_{\tilde I \subset I}  P_\omega^z(X_{T_{-L}} \in \tilde \Delta,T_{-L} \in \tilde I ) \min_{ m \in \tilde I }\min_{k \in[L, c L  ]} P_\omega^0(X_{n-m-k}=0)\\
  &\geq   C \frac{L^{\epsilon + d-1}}{n^{d/2}}\sum_{\tilde I \subset I}\min_{m \in \tilde I } \min_{k \in[L, c L  ]} P_\omega^0(X_{n-m-k}=0).
\end{align*}
To get a lower bound for $P_\omega^0(X_{n-m-k}=0)$, we can force the random walk to hit $\partial B_2(0,c_3)$, then successively make sojourns to $\Delta \setminus B_2(0,c_3)$ for at most $n-m-k$ times, and then return to the origin. This, with \eqref{eq: trap-ineq} and uniform ellipticity, yields 
\begin{align*}
P_\omega^0(X_{n-m-k}=0) \geq \kappa^{2 \sqrt{d}c_3} \left(\min_{x \in \partial B_2(0,c_3)} P_\omega^x(T_{B_2(0,c_3)} < T_{\Delta^c}) \right)^{n-m-k} \geq \kappa^{2 \sqrt{d}c_3} \left( 1- e^{-c_4L} \right)^{n-m-k} \\
\end{align*}
Applying this inequality, we get
\begin{align*}
\sum_{\tilde I \subset I } \min_{k \in[L, c L  ]} \min_{m \in \tilde I} P_\omega^0(X_{n-m-k} = 0) &\geq  \sum_{\tilde I \subset I} \min_{m \in \tilde I} \min_{k \in[L, c L  ]} \kappa^{2 \sqrt{d}c_3} \left( 1- e^{-c_4L} \right)^{n-m-k}\\
& \geq \sum_{\tilde I \subset I} \min_{m \in \tilde I}  \kappa^{2 \sqrt{d}c_3} \left( 1- e^{-c_4L} \right)^{n-m-L}.
\end{align*}
We choose the following partition for $I$: let 
\begin{align*}
\tilde I_j =\{ m \in \N: m \in [(n-L)/\langle v,e_1 \rangle +j\lfloor L^{\epsilon} \rfloor,(n-L)/\langle v,e_1 \rangle + (j+1)\lfloor L^{\epsilon} \rfloor)\}
\end{align*}
for $j \in J $ where $J$ is the corresponding index set. We thus have
\begin{align*}
 \sum_{\tilde I \subset I} \min_{m \in \tilde I}  \kappa^{2 \sqrt{d}c_3} \left( 1- e^{-c_4L} \right)^{n-m-L} &=  \sum_{j \in J} \min_{m \in \tilde I_j}  \kappa^{2 \sqrt{d}c_3} \left( 1- e^{-c_4L} \right)^{n-m-L} \\ 
 &\geq  \sum_{j = 1}^{\lfloor c n^{1/2}/L^\epsilon \rfloor}  \kappa^{2 \sqrt{d}c_3} \left( 1- e^{-c_4L} \right)^{ \frac{L}{\langle v,e_1 \rangle}-L + 2j L^\epsilon} \\
 & \geq C L^{-\epsilon} e^{\tilde{c} L}.
\end{align*}
Putting everything together, we get that for $\omega \in  F \cap  E_L$, for all $z \in B(-nv, cn^{1/2})$
\begin{align*}
P_\omega^z(X_n = x)  \geq C \frac{L^{d-1}}{n^{d/2}} e^{\tilde c L} = C\exp(\tilde c L-(d-1)\log L ) /n^{d/2} > e^{\hat{c}L}/n^{d/2},
\end{align*}
for some $\hat c >0$ for large $L$. Letting $L =\lceil \log u/ \hat{c} \rceil$, we conclude that 
\begin{align*}
P(\forall z \in B(-nv, cn^{1/2}), \: P_\omega^z(X_n = x) > u/n^{d/2}) \geq P( F \cap  E_{\lceil \log u/ \hat{c} \rceil}).
\end{align*}
 To finish the proof, we note that for any $L>0$, $ F$ and $ E_L$ are independent. Indeed, $ F$ is measurable with respect to $\{ \omega(x): \langle x,e_1 \rangle <   -L \}$ and $ E_L$ is measurable with respect to $\{\omega(x): x \in \Delta \}$. Since both sets do not intersect, and $P$ is an i.i.d. measure, we get $P( F \cap  E_{\lceil \log u/ \hat{c} \rceil}) = P( F)P( E_{\lceil \log u/ \hat{c} \rceil})$. Recall that $P( F) = 1- L^{-\xi(1)}> 1/2$ and $P( E_{\lceil \log u/ \hat{c} \rceil})> Ce^{-c(\log u)^d}$. Assuming $u$ is large enough such that $P( F) > 1/2$ finishes the proof.
\end{proof}
We can now finish the proof of Proposition \ref{prop: radon-approx}.
\begin{proof}[Proof of the lower bound in Proposition \ref{prop: radon-approx}]
For every $c'>0$ there exists a $C'>0$ such that
\begin{align*}
\left \{ \omega  \in \Omega : \sum_{z \in \Z^d}P_\omega^z(X_n = 0) > u \right\} \supset  \left\{\omega  \in \Omega : \forall z \in B(-nv, c'n^{1/2}), \: \: P_\omega^z(X_n = 0)  >C' u/n^{d/2} \right\}.
\end{align*}
Applying Lemma \ref{lem: smallbox-lower} for a small enough $c'$ finishes the proof.
\end{proof}

\noindent{\bf Acknowledgements.}
I would like to thank Ron Rosenthal for his support throughout all stages of this work, and Noam Berger for many interesting discussions. This reasearch was partially supported by ISF grant 771/17 and BSF grant 2018330.

\appendix
\section{Appendix}
\subsection{Proof of Proposition \ref{prop: heatkernel}} \label{sec: heatkernel-appendix}
The proof of Proposition \ref{prop: heatkernel} will be similar to the proof of \cite[Proposition 3.1]{berger2016local}, except ours will use the intersection estimates from Proposition \ref{prop: intersection} rather than their estimates. The proof is in three steps:
\begin{enumerate}
\item We first show in Lemma \ref{lem: box1} the quenched heat kernel concentrates around its average by using martingale techniques, which requires bounds for the number of intersection of two random walks. However, this method only works for the quenched probability of hitting large boxes.
\item In Lemma \ref{lem: box2} we use an induction argument to get a sub-optimal estimate for the probability of hitting boxes of all sizes.
\item To get sharp bounds for boxes of all sizes, in Lemma \ref{lem: box3} we rerun the argument in the first step, this time replacing the intersection estimates of Proposition \ref{prop: intersection} with the heat kernel bounds of Lemma \ref{lem: box2}. 
\end{enumerate} 

Fix $ j \in \N$ and define the event
\begin{align*}
B_N = B_N(j) = \left \{\forall 1 \leq k\leq N^2,  \tau_{k}-\tau_{k-1} \leq R_j(N)  \right \}.
\end{align*}
For $L \in \Z$, denote the hyperplane 
\begin{align*}
H_L = \{z \in \Z^d: \langle z, e_1 \rangle = L \}.
\end{align*} 
\begin{lemma} \label{lem: box1}
Let $d\geq 2$, and assume $P$ is uniformly elliptic, i.i.d. and satisfies condition $(T)$. For every $\theta \in (0,1]$, let $L(N,\theta)$ be the event that for every $\frac{2}{5} N^2 \leq M \leq N^2$, every $z \in \tilde {\mathcal P}(0,N)$, every $(d-1)$-dimensional cube $\Delta \subset H_M$ of side-length $N^\theta$ and every interval $I$ of length $N^\theta$
\begin{align*}
\vert P_\omega^z(X_{T_{M}}\in \Delta,T_{M} \in I, B_N )-\prob^z(X_{T_{M}}\in \Delta,T_{M} \in I, B_N ) \vert \leq CN^{\theta d}/ N^{d}.
\end{align*}
For every $\theta> (2d+1)/(2d+2)$ and $\alpha< 1 + \frac{d-1}{3d}$, there exist constants $C,c>0$ such that
\begin{align*}
P(L(N,\theta)) > 1- C\exp(-c(\log N)^{\alpha})
\end{align*}
for any $n \in \N$.
\end{lemma}
\begin{proof}
Fix $j \in \N$ and $\theta'\in (0,\theta)$, and define $V_N = \lfloor N^{2\theta'} \rfloor$. Let $\mathcal F$ be the $\sigma$-algebra generated by
\begin{align*}
\{\omega(x) : x \in \mathcal P^M(0,N) \},
\end{align*}
where $\mathcal{P}^M(0,N) = \mathcal P(0,N) \cap \{x \in \Z^d : \langle x,e_1 \rangle \leq M \}$. Let $ x_1, x_2, \ldots$ be any lexicographic ordering of the vertices in $\mathcal P^M(0,N)$, with the first coordinate being the most significant. Let $\mathcal F_i$ be the $\sigma$-algebra generated by $\{ \omega(x_1), \ldots, \omega(x_i) \}$.  Fix $N \in \N$, $y \in H_{M+V_N}$ and $m \in \N$, and define the martingale 
\begin{align*}
M_i = E[P_\omega^z(X_{T_{M+V_N}} = y, T_{M+V_N} = m, B_N) \vert \mathcal F_i],
\end{align*} 
as well as $U_i = \text{esssup}(\vert M_i - M_{i-1} \vert \vert \mathcal F_{k-1})$. We will bound the martingale differences, and then finish the proof by applying McDiarmid's inequality. We first show
\begin{align}\label{eq: mart-diff}
U_i \leq C R_j(N) V_N^{-(d+1)/2} E[P_\omega^z(x_i \: \text{is visited} , B_N) \vert  \mathcal F_i ].
\end{align}
Let $\mathtt{P}_i(\cdot) =  E[P_\omega^z(\cdot, B_N) \vert \mathcal F_{i}]$. The random walk under $\mathtt{P}_{i-1}$ has quenched transition probabilities on $\{x_1,\ldots, x_{i-1} \}$ and annealed transition probabilities on $\Z^d \setminus  \{x_1,\ldots, x_{i-1} \}$, while $\mathtt P_i$ has the same distribution except that on $x_i$ it has quenched transition probability. Hence both distributions are identical on the event the random walk never hits $x_i$, and so
\begin{alignat*}{2}
\vert M_i - M_{i-1} \vert &= \vert && \mathtt P_i(X_{T_{M+V_N}} = y, T_{M+V_N} = m) - \mathtt P_{i-1}(X_{T_{M+V_N}} = y, T_{M+V_N} = m)  \vert \\
&=   \vert &&   \mathtt P_i(X_{T_{M+V_N}} = y, T_{M+V_N} = m, x_i \: \text{is visited} ) \\
& &&- \mathtt P_{i-1}(X_{T_{M+V_N}} = y, T_{M+V_N} = m, x_i \: \text{is visited})  \vert .
\end{alignat*}
Until the first hitting time of $x_i$, $\mathtt P_i$ and $\mathtt P_{i-1}$ have the same distribution, and so we can couple them until they first hit $x_i$. Since we are under the event $B_n$, after at most time $R_j(n)$ after they hit $x_i$, both random walks will never backtrack. By definition of $\mathtt P_i$ and the ordering we chose for the vertices of $\mathcal P^M(0,N)$, the random walks after this time will have the annealed distribution. From these observations, we can apply the annealed derivative estimates from Lemma \ref{lem: ann-kernel1} to get
\begin{align*}
 U_i \leq \mathtt P_i(x_i \: \text{is visited})  C R_j(N) V_N^{-(d+1)/2} ,
\end{align*}
which yields \eqref{eq: mart-diff}. Let $B(x)= \{w \in H_{\langle x,e_1 \rangle - 1}: \norm{w-x}_\infty < R_j(N)  \}$. On the event $B_N$ and $x_i$ is visited, the random walk must first visit the hyperplane behind $x_i$ at $B(x_i)$. We thus have
\begin{align*}
\mathtt P_i(x_i \:  \text{is visited}) & =E[P_\omega^z(x_i \:  \text{is visited}, B_N) \vert \mathcal F_{i}]\\
&\leq \sum_{w \in B(x_i)}E[P_\omega^z(T_{\langle x_i, e_1 \rangle - 1} = w , B_N) \vert \mathcal F_{i}]\\
&= \sum_{w \in B(x_i)}P_\omega^z(T_{\langle x_i, e_1 \rangle - 1} = w , B_N) \\
&\leq  \sum_{w \in B(x_i)}P_\omega^z(w \: \text{is visited}, B_N) .
\end{align*}
Since $\vert B(x) \vert \leq C 2^d R^d_j(N)$, and every $w \in \Z^d$ is in $B(x)$ for at most $2^d R^d_j(N)$ points $x \in \Z^d$, we have
\begin{align*}
U \coloneqq \sum_{i} U_i^2 \leq  CR^{2d+2}_j(N)\sum_{i} P_\omega^z(x_i \: \text{is visited})^2 V_N^{-(d+1)}.
\end{align*}
Fix some $\epsilon>0$, and define the event 
\begin{align*}
A_N = \{\omega \in \Omega : E^z_\omega\otimes E^z_\omega[\mathcal I_{2N^2}] \leq N^{1 + \epsilon}  \},
\end{align*}
which by Proposition \ref{prop: intersection} satisfies $P(A_N^c) \leq Ce^{-c(\log N)^\alpha}$. For large enough $N$, 
\begin{align*}
\mathcal P(0,N) \subset \{x \in \Z^d: \norm{x}_1 \leq 2N^2\},
\end{align*} 
and so the number of intersection points in $\mathcal P(0,N)$ is bounded by $\mathcal I_{2N^2}$. Hence for $\omega \in A_N $, we have
\begin{align*}
U &\leq CR^{2d+2}_j(N)\sum_{i} P_\omega^z(x_i \: \text{is visited})^2 V_N^{-(d+1)} \\
&\leq CR^{2d+2}_j(N) V_N^{-(d+1)} E^z_\omega \otimes E^z_\omega  [\mathcal I_{2N^2}] \\
&\leq R^2_{j+1}(N) V_N^{-(d+1)} N^{1+\epsilon}.
\end{align*}
Applying McDiarmid's inequality, see \cite[Theorem 3.14]{mcdiarmid1998concentration}, we have
\begin{align*}
&P \bigg( \vert E[P_\omega^z(X_{T_{M+V_N}} = y,T_{M+V_N} = m, B_N) \vert \mathcal F] \\
&\qquad -\prob^z(X_{T_{M+V_N}} = y,T_{M+V_N} = m, B_N) \vert > N^{-d} , A_N^c\bigg) \\
&\leq  \exp\left(- \frac{N^{-2d}}{R^2_{j+1}(N) V_N^{-(d+1)}N^{1+\epsilon}} \right).
\end{align*}
For $\theta' > \theta $, we have
\begin{align*}
\exp\left(- \frac{N^{-2d}}{R^2_{j+1}(N) V_N^{-(d+1)}N^{1+\epsilon}} \right) \leq \exp\left(- \frac{N^{2\theta(d+1)+\delta}}{N^{2d+1}} \right) ,
\end{align*}
where $\delta>0$ can be taken arbitrarily small since $\theta'$ can be taken arbitrarily close to $\theta$. Note that this exponent is negligible as long as $\theta > (2d+1)/(2d+2)$. Hence if we let $W(N) \subset \Omega$ be the event
\begin{align*}
 \vert E[P_\omega^z(X_{T_{M+V_N}} = y,T_{M+V_N} = m, B_N) \vert \mathcal F] -\prob^z(X_{T_{M+V_N}} = y,T_{M+V_N} = m, B_N) \vert < N^{-d}
\end{align*}
for every $z \in \tilde {\mathcal P}(0,N)$, every $y \in H_{M+V_N}\cap \mathcal P(0,2N)$ and every $m \in \N$, we conclude that $P(W^c(N)) < e^{-c(\log N)^\alpha}$.  We can now continue exactly as in the proof of \cite[Lemma 3.5]{berger2016local}. The only estimates left to show are \cite[$(3.6)$]{berger2016local} and \cite[$(3.7)$]{berger2016local}, but these hold for any $\theta'<\theta$ with probability $C\exp(-c(\log n)^\alpha)$. This finishes the proof.
\end{proof}
\begin{lemma} \label{lem: box2}
Let $d \geq 2$ and assume $P$ is uniformly elliptic, i.i.d. and satisfies $(T)$. For every $\theta \in (0,1]$ and $h \in \N$, let $D^{(\theta,h)}(N)$ be the event that for every $z \in \tilde P(0,N)$, every $\frac{1}{2}N^2 \leq M \leq N^2$, every $(d-1)$-dimensional cube  $\Delta \subset H_M$ of side-length $N^\theta$ and every interval $I \subset \N$ of length $N^\theta$
\begin{align} \label{eq: sub-time}
P_\omega^z(X_{T_M} \in \Delta, T_M \in I) \leq R_h(N)N^{\theta d}/N^d.
\end{align}
Then for every $\theta \in (0,1]$, there exists $h = h(\theta)$ such that for every $\alpha < 1+ (d-1)/3d$, there exist constants $C,c>0$ such that  $P(D^{(\theta,h)}(N)) > 1 -C\exp(-c(\log n)^\alpha)$ for every $n \in \N$.
\end{lemma}
\begin{proof}
We will prove the lemma by descending induction on $\theta$. First note that if $(2d+1)/(2d+2)<\theta' <1$, then by Lemmas \ref{lem: ann-kernel2} and \ref{lem: box1} we have $P(D^{(\theta',h)}(N)) > 1 -C\exp(-c(\log n)^\alpha)$. For the inductive step, assume the statement of the lemma holds for some $\theta'$. Fix any $\theta$ satisfying $\theta' \frac{2d+1}{2d+2}<\theta<\theta'$, and let $h' = h(\theta')$ and $\rho =\theta/\theta' $. Define the event
\begin{equation*}
T(N,\rho) = \left\{\omega \in \Omega : \begin{alignedat}{2} \begin{gathered}  \forall v \in \tilde {\mathcal P}(0,N) \: P_\omega^v(X_{T_{\partial \mathcal P(v,[N^\rho])}} \not \in \partial^+ \mathcal P(v,[N^\rho])) < Ce^{-cR_5(N)}, \\[.25cm]  P_\omega^v(\vert T_{\partial \mathcal P(v,[N^\rho])} - \E^v[T_{\partial \mathcal P(v,[N^\rho])}] \vert > N^{\rho} R_j(N^{\rho})) \leq Ce^{-c(\log R_j(N))^{\alpha}/2} \end{gathered}\end{alignedat} \right \}.
\end{equation*}
Let $\varsigma$ be the time shift of the random walk, defined as $\varsigma_s(X_1,X_2,X_3,\ldots) = (X_{s+1}, X_{s+2}, \ldots)$, and also define the event
\begin{align*}
S(N) = D^{(\rho,1)}(N) \cap \bigcap_{\substack{z \in \mathcal P(0,2N) \\ s \in [-2NR_5(N),2NR_5(N)]}} \sigma_z \varsigma_s(D^{(\theta',h')}([N^\rho]))\cap T(N,\rho).
\end{align*}
In the proof of \cite[Lemma 3.6]{berger2016local}, the authors showed that $S(N) \subset D^{(\theta,h)}(N)$, and so we just have to bound $P(S^c(N))$. By the induction assumption, we only need to bound $P(T(N,\rho))$. By Lemma \ref{lem: exit-estimates} and Markov's inequality, we have for every $\alpha< 1+ (d-1)/3d$ and for every $j \in \N$
\begin{align*}
P(T(N,\rho)) \leq \exp(-c(\log N)^{\alpha(j+2)/(j+3)}/2).
\end{align*}
Letting $\alpha' < \alpha(j+2)/(j+3) < 1+ (d-1)/3d$, we have 
\begin{align*}
P(T(N,\rho)) \leq \exp(-c(\log N)^{\alpha'}/2).
\end{align*}
Since $j$ is arbitrary the result holds for any $\alpha' < 1 + (d-1)/3d$, which finishes the proof.
\end{proof}
\begin{lemma} \label{lem: box3} 
Let $d \geq 2$, and assume $P$ is uniformly elliptic, i.i.d. and satisfies $(T)$. Let $\mathcal F$ be the $\sigma$-algebra generated by $\{\omega(z): \langle z, e_1 \rangle \leq N^2 \}$. Let $\eta> 0$ such that $\eta <  2/(d-1)$, $V_N=[N^\eta]$ and define $R(N,\eta)$ to be the event that for every $z \in \tilde{\mathcal P}(0,N)$, every $y \in H_{N^2 + V_N}$ and every $m \in \N$
\begin{align*}
\vert E [ P_\omega^z(X_{T_{N^2 + V_N}}=y,T_{N^2 + V_N}=m )\vert \mathcal F ] -\prob^z(X_{T_{N^2 + V_N}}=y,T_{N^2 + V_N}=m ) \vert \leq CN^{-d}V_N^{- (d-1)/5}.
\end{align*}
Then for every $\alpha < 1+ (d-1)/3d$, there exist constants $C,c>0$ such that $P(R(N,\eta)) > 1- C\exp(-c(\log N)^\alpha)$ for every $n \in \N$.
\end{lemma}
\begin{proof} 
Fix some $\theta \in (0,1)$. Let $K$ be an integer such that $2^{-K-1}N^2 \leq V_N < 2^{-K}N^2$, and for $0 \leq k \leq K$, define
\begin{align*}
\mathcal P^{(k)} &=  \{x \in \mathcal P(0,N)  : 2^{-k-1} N^2 \leq N^2 - \langle x, e_1 \rangle < 2^{-k}N^2  \},\\
\mathcal P^{(0)} &= \{x \in \mathcal P(0,N)  : N^2/2 \leq N^2 - \langle x, e_1 \rangle   \},\\
F(y) &= \{ x \in \mathcal P(0,N): \norm{x- y -\vartheta \frac{\langle x-y,e_1 \rangle}{\langle \vartheta, e_1 \rangle }}_\infty \leq R_j(N)\vert \langle y-x,e_1 \rangle\vert^{1/2}\},\\
\mathcal P^{(k)}(y) &=  \mathcal P^{(k)} \cap F(y),
\end{align*}
see Figure \ref{fig: parallelogramparabola}.
\begin{figure}[t!]  \centering
    \includegraphics[scale=0.50]{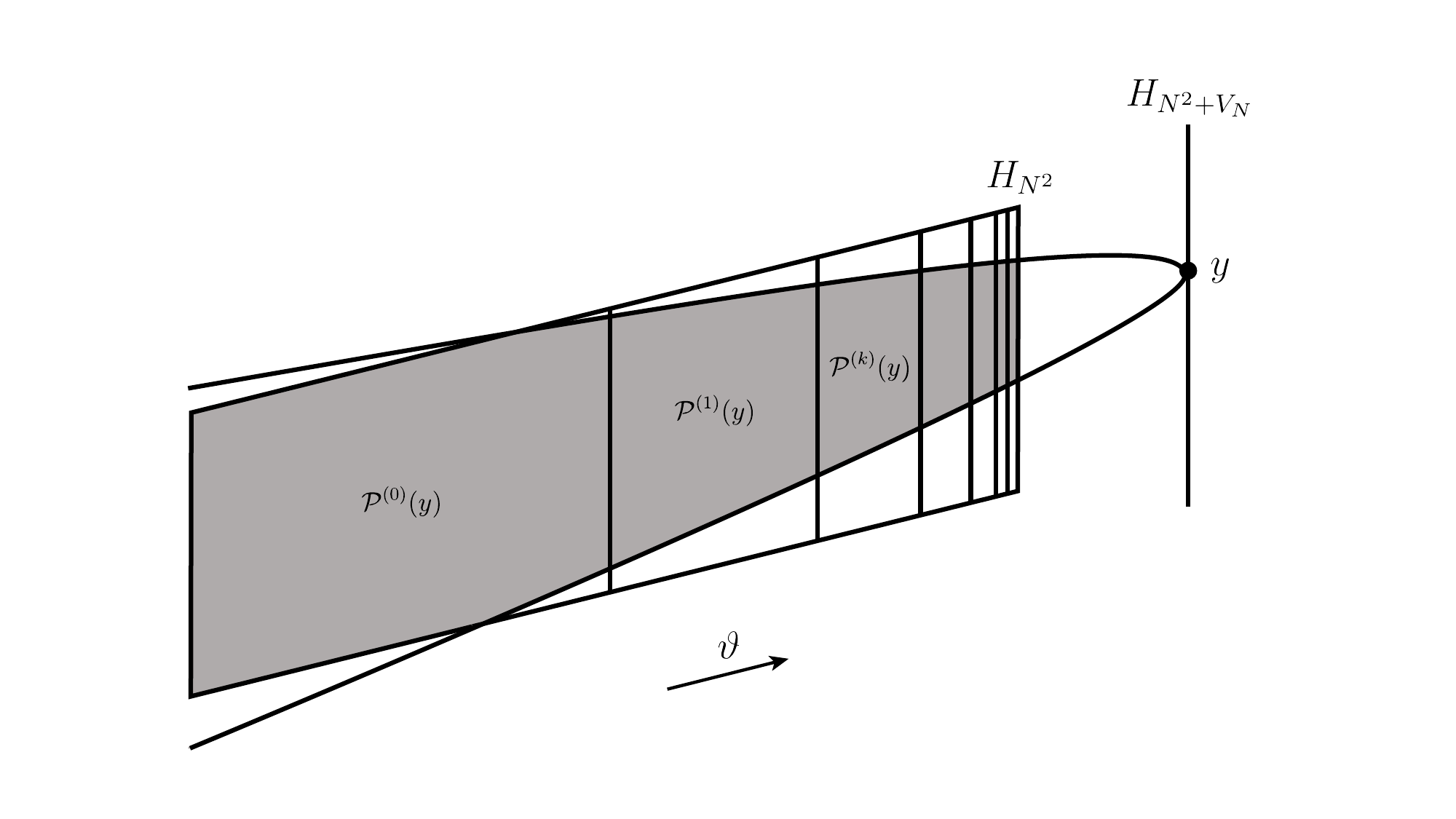}
    \caption{\textit{Since the random walk at time $n$ has variance of order $\sqrt{n}$, on the event the random walk hits $y$, it will typically only visit the shaded region.}}
    \label{fig: parallelogramparabola}
\end{figure}
 We repeat the argument from Lemma \ref{lem: box1}. Let $x_1, x_2, \ldots$ be a lexicographic ordering of the vertices of $\mathcal P(0,N)$ emphasizing the first coordinate. Let $\mathcal F_i$ be the $\sigma$-algebra generated by $\{\omega(x_1), \ldots ,\omega(x_i) \}$. Fix $N\in \N$, $y \in H_{N^2 + V_N}$ and $m \in \N$, and define the martingale
\begin{align*}
M_i = E [ P_\omega^z(X_{T_{N^2 + V_N}}=y,T_{N^2 + V_N}=m , B_N)\vert \mathcal F_{i} ],
\end{align*}
as well as $U_i = \text{esssup}(\vert M_i - M_{i-1} \vert \vert \mathcal F_{i-1}) $. As in the proof of Lemma \ref{lem: box1}, we have
\begin{align*}
U_i \leq \vert \mathtt P_{i}(X_{T_{M+V_N}} = y, T_{M+V_N} = m, x_i \: \text{is visited} ) - \mathtt P_{i-1}(X_{T_{M+V_N}} = y, T_{M+V_N} = m,x_i \: \text{is visited})  \vert,
\end{align*}
where $\mathtt{P}_i(\cdot) =  E[P_\omega^z(\cdot, B_N) \vert \mathcal F_{i}]$. For $x \in \mathcal P(0,N)$, let $L_x=N^2 + V_N - \langle x,e_1 \rangle $ and define the interval
\begin{align*}
I(x,m) =  [m - \E^z[T_{L_x} ] - R_j(L_x)L_x^{1/2},m - \E^z[T_{L_x} ] + R_j(L_x)L_x^{1/2}].
\end{align*}
We now bound $P_\omega^z(X \: \text{visits} \: x \: \text{in} \: I(x,m)^c, T_{N^2 + V_N}= m)$. On the event the random walk visits $x \in H_{N^2 + V_N - L}$ and hits $H_{N^2 + V_N}$ at time $m$, with high probability it will visit $x$ at some time in $I(x,m)$. More explicitly, 
\begin{align*}
&P_\omega^z(X \: \text{visits} \: x \: \text{in} \: I(x,m)^c, T_{N^2 + V_N}= m,B_N) \\
&=P_\omega^z(X \: \text{visits} \: x \: \text{in} \: I(x,m)^c\cap [m], T_{N^2 + V_N}= m,B_N) \\
&\leq \sum_{k  \in I(x,m)^c\cap [m]} P_\omega^z(X_k = x)P_\omega^x(T_{N^2 + V_N}= m-k) \\
&\leq \sum_{k  \in I(x,m)^c \cap [m]}P_\omega^z(X_k = x)P_\omega^x(  \vert T_{L_x} - \E^x [T_{L_x}] \vert > R_j(L_x)L_x^{1/2} ) \\
&\leq P_\omega^x(  \vert T_{L_x} - \E^x [T_{L_x}] \vert >C R_j(L_x)L_x^{1/2} ),
\end{align*}
where the first equality follows from the observation that on $B_N$, the distance traveled between two regeneration times is bounded by $R_j(N) < V_N$, and so on the event the random walk hits $x$, it must happen before it reaches $H_{N^2+V_N}$. From this calculation, we have by Markov's inequality
\begin{align*}
&P \left (P_\omega^z(X \: \text{visits} \: x \: \text{in} \: I(x,m)^c, T_{N^2 + V_N}= m) >C\exp(-c(\log R_j(N))^\alpha/2)   \right) \\
&\leq P \left ( P_\omega^x(  \vert T_{L_x} - \E^x [T_{L_x}] \vert >C R_j(L_x)L_x^{1/2} )>C\exp(-c(\log R_j(N))^\alpha/2)   \right) \\
&\leq \prob^x(  \vert T_{L_x} - \E^x [T_{L_x}] \vert >C R_j(L_x)L_x^{1/2} )C^{-1}\exp(c(\log R_j(N))^\alpha/2) \\
&\leq \tilde C \exp(-c(\log R_j(N))^\alpha/2),
\end{align*}
where the last inequality follows from Lemma \ref{lem: exit-estimates} and that $L_{x} \geq V_N $ for every $x \in \mathcal{P}(0,N)$.
Define the event
\begin{equation*}
G_N =  \left\{\omega \in \Omega: \begin{alignedat}{2}  \begin{gathered} \forall x \in \mathcal P(0,N),\forall z \in \tilde{\mathcal P}(0,N), \forall m \in [0, CN^2], \\[.25cm]
 P_\omega^z(X \: \text{visits} \: x \: \text{in} \: I(x,m)^c, T_{N^2 + V_N}= m, B_N) < C\exp(-c(\log R_j(N))^\alpha/2) \end{gathered} \end{alignedat} \right \}.
\end{equation*}
 Then by the union bound, $P(G_N^c)  < C \exp(-c(\log R_j)^\alpha/2) $. For $\omega \in G_N$, we have 
\begin{align*}
U_i \leq \vert &\mathtt P_{i}(X_{T_{M+V_N}} = y, T_{M+V_N} = m, x_i \: \text{is visited in}\: I(x_i,m) ) \\
&- \mathtt P_{i-1}(X_{T_{M+V_N}} = y, T_{M+V_N} = m,x_i \: \text{is visited in}\: I(x_i,m))  \vert + N^{-\xi(1)}.
\end{align*}
We bound $U_i$ using the same argument as in the proof of Lemma \ref{lem: box1}: until the first hitting time of $x_i$, $\mathtt P_i$ and $\mathtt P_{i-1}$ have the same distribution, and so we can couple them until the first hitting time of $x_i$. Under the event $B_n$, the regeneration times are bounded by $R_j(N)$, which combined with estimates from Lemma \ref{lem: ann-kernel1} gives us for $x_i \in F(y)$
\begin{align*}
U_i \leq \mathtt P_i( x_i \: \text{is visited in}\: I(x_i,m) ) \cdot C R_j(N)L_{x_i}^{-(d+1)/2}.
\end{align*}
To bound $U_i$ when $x_i \not \in F(y)$, we define the event
\begin{align*}
J_N =  \left\{\omega \in \Omega: \begin{alignedat}{2}  \begin{gathered} \forall z \in \tilde{\mathcal P}(0,N), \forall n \in [V_N, \infty), \\[.25cm]
P_\omega^z(\norm{X_n - \E^z[X_n]}_\infty > R_5(n) n^{1/2}) \leq C\exp(-c R_5(n))\end{gathered} \end{alignedat} \right \}.
\end{align*}
From a union bound, Markov's inequality, and Lemma \ref{lem: exit-ball}, we get that $P(J_N^c) \leq  C\exp(-c R_5(V_N))$. For $\omega \in J_N$ and $x \not \in F(y)$, by definition of $F(y)$ we have 
\begin{align*}
P_\omega^x(X_{T_{N^2 + V_N}} = y) \leq P_\omega^x( \exists n > V_N, \norm{X_n - \E^z[X_n]}_\infty > R_5(n) n^{1/2}) = N^{-\xi(1)}.
\end{align*}
Hence for $\omega \in J_N$, $U_i = N^{-\xi(1)}$ when $x_i \not \in F(y)$. Applying these bounds, as well as similar arguments from Lemma \ref{lem: box1}, we have for $\omega \in J_N$
\begin{align*}
U \coloneqq \sum_{i} U_i^2 & \leq \sum_{k=0}^K \sum_{x_i \in \mathcal P^{(k)}(v)} P_\omega^z( x_i \: \text{is visited in}\: I(x_i,m))^2 \frac{C R^{2d+2}_j }{L_{x_i}^{d+1}}  + N^{-\xi(1)}.
\end{align*}
To bound $U$, define the event $E_N = D^{(\theta,h)} \cap \{\omega \in \Omega: E^z_\omega \otimes E^z_\omega [\mathcal I_{2N^2}] < N^{1+ \epsilon} \}$, which for appropriate $h$ satisfies $P(E^c_N) \leq C\exp(-c(\log N)^\alpha)$ by Proposition \ref{prop: intersection} and Lemma \ref{lem: box2}. We first bound the $k=0$ summand. For $x_i \in \mathcal P^{(0)}(v)$, $L_{x_i} \geq N^2/2$ and so for $\omega \in E_N$ we have
\begin{align*}
&\sum_{x_i \in \mathcal P^{(0)}(v)} P_\omega^z( x_i \: \text{is visited in}\: I(x_i,m))^2 \frac{C R^{2d+2}_j }{L_{x_i}^{d+1}}\\
&\leq CR_j^{2d+2}(N) N^{-2(d+1)}   \sum_{x_i \in \mathcal P^{(0)}(v)} P_\omega^z(x_i \: \text{is visited})^2\\
& \leq CR_j^{2d+2}(N)  N^{-2(d+1)}  E^z_\omega \otimes E^z_\omega[\mathcal I_{2N^2}] \\
&\leq C R_{j+1}(N) N^{-2d-1+ \epsilon}.
\end{align*}
Note that when $x_i \in \mathcal P^{(k)}(v)$ for $k>0$, we have $V_N + 2^{-k-1}N^2 \leq L_{x_i} < V_N + 2^{-k}N^2 $. Hence when $k>0$, we have for $\omega \in  E_N$ 
\begin{align*}
&\sum_{x_i \in \mathcal P^{(k)}(v)} P_\omega^z( x_i \: \text{is visited in}\: I(x_i,m))^2 \frac{C R^{2d+2}_j }{L_{x_i}^{d+1}}  \\
&\leq \frac{C R^{2d+2}_j }{(V_N+N^2 2^{-k-1})^{d+1}} \sum_{x_i \in \mathcal P^{(k)}(v)} P_\omega^z( x_i \: \text{is visited in}\: I(x_i,m))^2 \\
&\leq \frac{C R^{2d+2}_j }{(V_N+N^2 2^{-k-1})^{d+1}} \sum_{x_i \in \mathcal P^{(k)}(v)} \frac{ R_h(N) N^{2 \theta d}L_{x_i}}{N^{2d}} \\
&\leq \frac{C R^{2d+2}_j R_h(N) }{(V_N+N^2 2^{-k-1})^{d+1}} \frac{ N^{2 \theta d}(V_N + N^2 2^{-k})}{N^{2d}}  \vert \mathcal P^{(k)}(v) \vert \\
&\leq \frac{C R^{2d+2}_j R_h(N) }{(V_N+N^2 2^{-k})^{d}} \frac{ N^{2 \theta d}}{N^{2d}} \vert \mathcal P^{(k)}(v) \vert  \\
&\leq \frac{C R^{2d+2}_j R_h(N) }{(V_N+N^2 2^{-k})^{d}}  N^{d+1}2^{-k(d+1)/2} \frac{ N^{2 \theta d}}{N^{2d}} \\
&\leq \frac{C R^{2d+2}_j  R_h(N)}{(V_N+N^2 2^{-k})^{(d-1)/2}}  \frac{ N^{2 \theta d}}{N^{2d}}.
\end{align*}
Putting both estimates together, we conclude that for $\omega \in J_N \cap E_N \cap G_N$
\begin{align*}
U &\leq C R_{j+1}(N) N^{-2d-1+ \epsilon} + \sum_{k=1}^K \frac{C R^{2d+2}_j }{(V_N+N^2 2^{-k})^{(d-1)/2}}  \frac{ N^{2 \theta d}}{N^{2d}} \\
&\leq C R_{j+1}(N) N^{-2d-1+ \epsilon} + \frac{C R_{\max \{j+1,h+1 \}} }{V_N^{(d-1)/2}}  \frac{ N^{2 \theta d}}{N^{2d}} \\
&\leq \frac{C  R_{\max \{j+1,h+1 \}} }{V_N^{(d-1)/2}}  N^{-2d},
\end{align*}
where the last inequality follows from assuming $\eta < 2/(d-1)$, and that $\theta$ can be chosen arbitrarily close to $0$. We now apply McDiarmid's inequality 
\begin{align*}
  &P \bigg( \vert E [ P_\omega^z(X_{T_{N^2 + V_N}}=y,T_{N^2 + V_N}=m ,B_N)\vert \mathcal F ] \\
 &\qquad  -\prob^z(X_{T_{N^2 + V_N}}=y,T_{N^2 + V_N}=m, B_N ) \vert > CN^{-d}V_N^{- (d-1)/5} \bigg) \\
& \leq \exp \left(-c\frac{N^{-2d}V_N^{- 2(d-1)/5}}{N^{-2d} V_N^{-(d-1)/2} }\right)  +P((J_N \cap E_N\cap G_N)^c)  \\
& \leq C\exp(-c(\log N)^{\alpha(j+2)/(j+3)}).
\end{align*}
Lemma \ref{lem: sub-slowdown} implies $\prob(B^c_N) \leq  C\exp(-c(\log R_j(N))^\alpha) $ for any $\alpha<1+ (d-1)/3d$, and so by Markov's inequality  
\begin{align*}
P\left(P_\omega(B_N^c) > C\exp(-c(\log R_j(N))^\alpha) /2 \right ) \leq C\exp(-c(\log R_j(N))^\alpha/2) .
\end{align*}
With this estimate, we conclude that $P(R(N,\eta))>1-\exp(-c(\log N)^{\alpha})$ for any $\alpha<1+ (d-1)/3d$. 
 \end{proof}
 \begin{proof}[Proof of Proposition \ref{prop: heatkernel}]
 Continue as in the proof of \cite[Proposition 3.1]{berger2016local}, replacing in \cite[Lemma 3.7]{berger2016local} with Lemma \ref{lem: box3}.
 \end{proof}
\subsection{Proof of Theorem \ref{thm: slowdown}} \label{sec: slowdown-appendix} 
We use this intersection result to provide estimates for the mean of the quenched exit distribution.
 \begin{lemma} \label{lem: quenched-mean}
Let $d \geq 3$ and assume $P$ is uniformly elliptic, i.i.d. and satisfies condition $(T)$. Fix $\epsilon>0$, and let $K(N,\epsilon)$ be the event that for every $z \in \tilde P(0,N)$,
\begin{align*}
\norm{E_\omega^z[X_{T_{\partial P(0,N)}}]-\E^z[X_{T_{\partial P(0,N)}}]}_{\infty} \leq N^\epsilon.
\end{align*}
Then $P(K(N,\epsilon)) = 1-N^{-\xi(1)}$.
\end{lemma}
\begin{proof}
Rerun the same proof of \cite[Lemma 4.11]{berger2012slowdown}, replacing the intersection bounds \cite[Lemma 4.11]{berger2012slowdown} with Lemma \ref{lem: intersections-upgrade}.
\end{proof}
With this lemma, we are ready to prove Theorem \ref{thm: slowdown}.
\begin{proof}[Proof of Theorem \ref{thm: slowdown}]
In \cite{berger2012slowdown}, Berger proved Theorem \ref{thm: slowdown} for $d \geq 4$. The only part of his proof that depends on $d$ is \cite[Proposition 4.5]{berger2012slowdown}. Lemma \ref{lem: exit-estimates} and Proposition \ref{prop: heatkernel} implies \cite[Proposition 4.5(1)]{berger2012slowdown} and \cite[Proposition 4.5(3)]{berger2012slowdown} for $d \geq 3$, respectively, and so we only need to prove \cite[Proposition 4.5(2)]{berger2012slowdown}:
\begin{align*}
P \left( \norm{E_\omega^z[X_{T_{\partial P(0,N)}}]-\E^z[X_{T_{\partial P(0,N)}}]}_{\infty} \leq R_3(N) \right) > 1- N^{-\xi(1)},
\end{align*}
 which is sharper than our estimate Lemma \ref{lem: quenched-mean}. We now explain why for the proof of the theorem, Lemma \ref{lem: quenched-mean} suffices.

Let $\overline D$ and $\overline{ \mathbb D}$ be the quenched and annealed exit distribution from $\mathcal P(0,N)$ conditioned on exiting through $\partial^+\mathcal P(0,N)$, respectively.  \cite[Proposition 4.5]{berger2012slowdown} is only used in order to show $\overline D$ and $\overline {\mathbb D}$ are $(N^{-\theta \frac{d-1}{2(d+1)}},(d+1)N^\theta)$-close for any $\theta \in (0,1/2)$-- see \cite[Definition 11]{berger2012slowdown} for the definition of two distributions to be $(\lambda,k)$-close for $\lambda < 1$ and $k \in \N$, and the remarks after the definition. Reviewing the proof, to prove that these two distribution are $(N^{-\theta \frac{d-1}{2(d+1)}},(d+1)N^\theta)$-close, it is enough to show for all $\epsilon>0$
\begin{align*}
P \left(  \norm{E_\omega^z[X_{T_{\partial P(0,N)}}]-\E^z[X_{T_{\partial P(0,N)}}]}_{\infty} \leq N^\epsilon \right) > 1- N^{-\xi(1)}.
\end{align*}
Hence Lemma \ref{lem: quenched-mean} is adequate for our purposes, and we are done with the proof.
\end{proof}
\subsection{Proof of Lemma \ref{lem: d-box}} \label{sec: d-box-appendix}
Combining Theorem \ref{thm: slowdown} and \eqref{eq: berger-slowdown}, we have for $d \geq 3$
\begin{align} \label{eq: slowdown-better}
\forall \alpha<d, \: \prob^0(\tau_1 > u) < Ce^{-(\log u)^{\alpha}}.
\end{align} 
Rerunning the proof of Proposition \ref{prop: intersection}, replacing Lemma \ref{lem: sub-slowdown} with \eqref{eq: slowdown-better}, we have for $d \geq 3$
\begin{align}\label{eq: intersection-better}
\forall \alpha<d, \: P\left(E^0_\omega\otimes E^0_\omega  [\mathcal I_n] > n^{1/2+\epsilon}  \right) < C\exp(-c (\log n)^\alpha).
\end{align}
 Recall the event $F(n)$ from Proposition \ref{prop: heatkernel}. Going over the proof of this proposition, we see that $P(F(n)^c)$ is dominated by the tail estimates of the number of intersections and regeneration times. Hence we can replace Lemma \ref{lem: sub-slowdown} and Proposition \ref{prop: intersection} with \eqref{eq: slowdown-better} and \eqref{eq: intersection-better} to get $P(F(n)^c)\leq C\exp(-c (\log n)^\alpha)$ for every $\alpha<d$.

Finally, we can rerun the arguments in the proof of \cite[Proposition 4.1]{berger2012slowdown}, replacing their quenched heat kernel bounds with Proposition \ref{prop: heatkernel}. Tracking their proof, we see that $P(E(n)^c)$ is dominated by the probability of the events in Lemmas \ref{lem: sub-slowdown} and \ref{lem: exit-ball}, and $F(n)^c$. We conclude that $P(E(n)^c)$ has the desired upper bound. 
\bibliographystyle{acm}
\bibliography{references}
\end{document}